\newif \ifCMAM \CMAMfalse

\ifCMAM
\documentclass[USenglish]{article}
\usepackage[utf8]{inputenc}
\usepackage[small]{dgruyter}
\usepackage{microtype}
\else
\documentclass[a4paper,12pt]{article}
\usepackage{a4wide}
\usepackage[latin1]{inputenc}
\usepackage[T1]{fontenc}
\usepackage{authblk}

\fi

\usepackage{amsmath,amssymb,amsfonts,amsthm}
\usepackage{mathabx,stmaryrd}
\usepackage{epsfig}
\usepackage{bbold}
\usepackage{paralist}
\usepackage[normalem]{ulem}
\usepackage{nicefrac}
\usepackage{xspace}

\ifCMAM
\else
\usepackage{hyperref}
\hypersetup{pdfstartview=XYZ}
\hypersetup{
  unicode = true,                           
  pdftoolbar = true,                        
  pdfmenubar = true,                        
  pdffitwindow = true,                      
  pdfauthor = {Cicuttin/Ern/Lemaire},       
  pdftitle  = {msHHO},                      
  pdfnewwindow = true,                      
  colorlinks = true,                        
  linkcolor = blue,                         
  citecolor = blue,                         
  filecolor = blue,                         
  urlcolor = blue                           
}
\fi

\ifCMAM
\else
\usepackage{tikz}
\usepackage{caption}
\usepackage{pgfplots,pgfplotstable}
\pgfqkeys{/pgfplots}{
  cycle list name = black white
}
\pgfplotsset{compat=newest}
\fi

\newcommand{\email}[1]{\href{mailto:#1}{#1}}

\newcommand{\define}{\mathrel{\mathop:}=}
\newcommand{\enifed}{=\mathrel{\mathop:}}

\newcommand{\eps}{\varepsilon}

\newcommand{\R}{\mathbb{R}}
\newcommand{\Z}{\mathbb{Z}}
\newcommand{\N}{\mathbb{N}}

\renewcommand{\vec}[1]{\boldsymbol{#1}}
\newcommand{\grad}{\vec{\nabla}}
\renewcommand{\div}{{\rm div}}

\renewcommand{\L}[1][\Omega]{L^2(#1)}
\newcommand{\Hoz}{H^1_0(\Omega)}
\newcommand{\HoPer}{H^1_{\rm per}(Q)}

\newcommand{\Leps}{{\cal L}^1_\eps}
\newcommand{\Reps}{{\cal R}_\eps}

\renewcommand{\TH}{{\cal T}_H}
\newcommand{\FH}{{\cal F}_H}
\newcommand{\FHi}{{\cal F}^{\rm i}_H}
\newcommand{\FHb}{{\cal F}^{\rm b}_H}
\newcommand{\FT}{{\cal F}_T}

\newcommand{\PikT}{\Pi^k_T}
\newcommand{\PikF}{\Pi^k_F}

\newcommand{\VkpoT}{V^{k{+}1}_{\eps,T}}

\newcommand{\UkT}{\underline{{\rm U}}^k_T}
\newcommand{\UkH}[1][H]{\underline{{\rm U}}^k_{#1}}
\newcommand{\Pkd}[1]{\mathbb{P}^k_d(#1)}
\newcommand{\Pkmd}[1]{\mathbb{P}^{k-1}_d(#1)}
\newcommand{\Pkpd}[1]{\mathbb{P}^{k+1}_d(#1)}
\newcommand{\Pkdmo}[1]{\mathbb{P}^k_{d{-}1}(#1)}

\newcommand{\uvT}{\underline{{\rm v}}_T}
\newcommand{\vT}{{\rm v}_T}
\newcommand{\vpT}{{\rm v}_{\FT}}
\newcommand{\vF}{{\rm v}_F}
\newcommand{\uvH}{\underline{{\rm v}}_H}
\newcommand{\vH}{{\rm v}_{\TH}}
\newcommand{\vpH}{{\rm v}_{\FH}}
\newcommand{\uuT}{\underline{{\rm u}}_T}
\newcommand{\uueT}{\underline{{\rm u}}_{\eps,T}}
\newcommand{\ueT}{\underline{{\rm e}}_{\eps,T}}
\newcommand{\ueH}{\underline{{\rm e}}_{\eps,H}}
\newcommand{\uT}{{\rm u}_T}

\newcommand{\uuH}{\underline{{\rm u}}_H}

\newcommand{\IkT}{\underline{{\rm I}}^k_T}

\newcommand{\senorm}[2]{{\left|#2\right|}_{#1}}
\newcommand{\norm}[2]{{\left\|#2\right\|}_{#1}}

\newcommand{\term}{\mathfrak{T}}

\newtheorem{theorem}{Theorem}[section]
\newtheorem{lemma}[theorem]{Lemma}
\newtheorem{proposition}[theorem]{Proposition}

\newtheorem{remark}[theorem]{Remark}
\newtheorem{definition}[theorem]{Definition}

\ifCMAM
\else
\title{A Hybrid High-Order method for highly oscillatory elliptic problems}
\author[1]{Matteo Cicuttin\footnote{\email{matteo.cicuttin@enpc.fr}}}
\author[1]{Alexandre Ern\footnote{\email{alexandre.ern@enpc.fr}}}
\author[2]{Simon Lemaire\footnote{Corresponding author:~\email{simon.lemaire@inria.fr}}}
\affil[1]{Universit\'e Paris-Est, CERMICS (ENPC),
          6--8 avenue Blaise Pascal,
          77455 Marne-la-Vall\'ee Cedex 2, and Inria Paris, 2 rue Simone Iff, 75012 Paris (France)}
\affil[2]{\'Ecole Polytechnique F\'ed\'erale de Lausanne, FSB-MATH-ANMC,
          Station 8,
          1015 Lausanne (Switzerland), and Inria Lille - Nord Europe, 40 avenue Halley, 59650 Villeneuve d'Ascq (France)}
\fi

\begin{document}

\ifCMAM
\author[1]{Matteo Cicuttin}
\author[1]{Alexandre Ern}
\author*[2]{Simon Lemaire} 
\affil[1]{Universit\'e Paris-Est, CERMICS (ENPC),
          6--8 avenue Blaise Pascal,
          77455 Marne-la-Vall\'ee Cedex 2, and Inria Paris, 2 rue Simone Iff, 75012 Paris (France); emails:~\email{matteo.cicuttin@enpc.fr},~\email{alexandre.ern@enpc.fr}.}
\affil[2]{\'Ecole Polytechnique F\'ed\'erale de Lausanne, FSB-MATH-ANMC,
          Station 8,
          1015 Lausanne (Switzerland), and Inria Lille - Nord Europe, 40 avenue Halley, 59650 Villeneuve d'Ascq (France); email:~\email{simon.lemaire@inria.fr}.}
\title{A Hybrid High-Order method for highly oscillatory elliptic problems}
\runningtitle{A HHO method for highly oscillatory elliptic problems}
\abstract{We devise a Hybrid High-Order (HHO) method for highly oscillatory elliptic problems that is capable of handling general meshes. The method hinges on discrete unknowns that are polynomials attached to the faces and cells of a coarse mesh; those attached to the cells can be eliminated locally using static condensation. The main building ingredient is a reconstruction operator, local to each coarse cell, that maps onto a fine-scale space spanned by oscillatory basis functions.
The present HHO method generalizes the ideas of some existing multiscale approaches, while providing the first complete analysis on general meshes. It also improves on those methods, taking advantage of the flexibility granted by the HHO framework. The method handles arbitrary orders of approximation $k\geq 0$. For face unknowns that are polynomials of degree $k$, we devise two versions of the method, depending on the polynomial degree $(k-1)$ or $k$ of the cell unknowns. We prove, in the case of periodic coefficients, an energy-error estimate of the form $\left(\eps^{\nicefrac12}+H^{k+1}+(\eps/H)^{\nicefrac12}\right)$, and we illustrate our theoretical findings on some test-cases.}
\communicated{...}
\received{...}
\accepted{...}
\journalname{Comput. Meth. Appl. Math.}
\journalyear{2018}
\journalvolume{..}
\journalissue{..}
\startpage{1}
\aop
\DOI{...}
\else
\fi

\maketitle

\ifCMAM
\else
\begin{abstract}
We devise a Hybrid High-Order (HHO) method for highly oscillatory elliptic problems that is capable of handling general meshes. The method hinges on discrete unknowns that are polynomials attached to the faces and cells of a coarse mesh; those attached to the cells can be eliminated locally using static condensation. The main building ingredient is a reconstruction operator, local to each coarse cell, that maps onto a fine-scale space spanned by oscillatory basis functions.
The present HHO method generalizes the ideas of some existing multiscale approaches, while providing the first complete analysis on general meshes. It also improves on those methods, taking advantage of the flexibility granted by the HHO framework. The method handles arbitrary orders of approximation $k\geq 0$. For face unknowns that are polynomials of degree $k$, we devise two versions of the method, depending on the polynomial degree $(k-1)$ or $k$ of the cell unknowns. We prove, in the case of periodic coefficients, an energy-error estimate of the form $\left(\eps^{\nicefrac12}+H^{k+1}+(\eps/H)^{\nicefrac12}\right)$, and we illustrate our theoretical findings on some test-cases.
\end{abstract}
\fi

\section{Introduction}

Over the last few years, many advances have been accomplished in the design of arbitrary-order polytopal discretization methods. Such methods are capable of handling meshes with polytopal cells, and possibly including hanging nodes. The use of polytopal meshes can be motivated by the increased flexibility, when meshing complex geometries, or when using agglomeration techniques for mesh coarsening (see, e.g.,~\cite{BBCDP:12}). 
Classical examples of polytopal methods are the (polytopal) Finite Element Method (FEM)~\cite{Wachs:75,SukTa:04}, which typically uses non-polynomial basis functions to enforce continuity, and non-conforming methods such as the Discontinuous Galerkin (DG)~\cite{ArBCM:02,DPErn:12,CanGH:14} and the Hybridizable Discontinuous Galerkin (HDG)~\cite{CocGL:09} methods.  
We also mention the Weak Galerkin (WG)~\cite{WanYe:13} method (see~\cite{Cockb:16} for its links to HDG).

More recently, new paradigms have emerged. One salient example is the Virtual Element Method (VEM)~\cite{BBCMM:13}, which is formulated in terms of virtual (i.e., non-computed) conforming functions. The key idea is that the virtual space contains those polynomial functions leading to optimal approximation properties, whereas the remaining functions need not be computed (only their degrees of freedom need to be) provided some suitable local stabilization is introduced. The degrees of freedom in the VEM are attached to the mesh vertices, and, as the order of the approximation is increased, also to the mesh edges, faces, and cells. Another recent polytopal method is the Hybrid High-Order (HHO) method, which has been introduced for locking-free linear elasticity in \cite{DPErn:15a}, and for diffusion in~\cite{DPELe:14}. The HHO method has been originally formulated as a non-conforming method, using polynomial unknowns attached to the mesh faces and cells. The HHO method has been bridged in~\cite{CoDPE:16} both to HDG (by identifying a suitable numerical flux trace), and to the non-conforming VEM considered in~\cite{AyuLM:16} (by identifying an isomorphism between the HHO degrees of freedom and a local virtual finite-dimensional space, which again contains those polynomial functions leading to optimal approximation properties).
The focus here is on HHO methods. HHO methods offer several assets, including a dimension-independent construction, local conservativity, and attractive computational costs, especially in 3D. Indeed, the HHO stencil is more compact than for methods involving degrees of freedom attached to the mesh vertices, and static condensation allows one to eliminate cell degrees of freedom, leading to a global problem expressed in terms of face degrees of freedom only, whose number grows quadratically with the polynomial order, whereas the growth of globally coupled degrees of freedom is typically cubic for DG methods.

In this work, we are interested in elliptic problems featuring heterogeneous/anisotropic coefficients that are highly oscillatory.
The case of slowly varying coefficients has already been treated in~\cite{DPErn:15b,DPELe:16}, where error estimates tracking the dependency of the approximation with respect to the local heterogeneity/anisotropy ratios have been derived.
Let $\Omega$ be an open, bounded, connected polytopal subset of $\R^d$, $d\in\{2,3\}$, and $\eps>0$, supposedly much smaller than the diameter of the domain $\Omega$, encode the highly oscillatory nature of the coefficients. We consider the model problem
\begin{equation} \label{eq:osc}
  \left\{
  \begin{alignedat}{3}
    -\div(\mathbb{A}_\eps\grad u_\eps)&=f&\quad&\text{in $\Omega$}, \\
    u_\eps&=0&\quad&\text{on $\partial\Omega$},
  \end{alignedat}
  \right.
\end{equation}
where $f\in\L$ is non-oscillatory, and $\mathbb{A}_\eps$ is an oscillatory, uniformly elliptic and bounded matrix-valued field on $\Omega$. It is well-known that the $H^{k+2}$-norm of the solution $u_\eps$ to Problem~\eqref{eq:osc} scales as $\eps^{-(k+1)}$, meaning that monoscale methods (including the monoscale HHO method of order $k\geq 0$ of~\cite{DPErn:15b,DPELe:16}) provide an energy-norm decay of the error of order ${(h/\eps)}^{k+1}$. To be accurate, such methods must hence rely on a mesh resolving the fine scale, i.e. with size $h\ll\eps$. Since $\eps$ is supposedly much smaller than the diameter of $\Omega$, an accurate approximation necessarily implies an overwhelming number of degrees of freedom. In a multi-query context, where the solution is needed for a large number of right-hand sides (e.g., an optimization loop, with $f$ as a control and~\eqref{eq:osc} as a distributed constraint), a monoscale solve is hence unaffordable. In that context, multiscale methods may be preferred. Multiscale methods aim at resolving the fine scale in an offline step, reducing the online step to the solution of a system of small size, based on oscillatory basis functions computed in the offline step, on a coarse mesh with size $H\gg\eps$. In a single-query context, multiscale methods are also interesting since they allow one to organize computations in a more efficient way.

Multiscale approximation methods on classical element shapes (such as simplices or quadrangles/hexahedra) have been extensively analyzed in the literature. Examples include, e.g., the multiscale Finite Element Method (msFEM)~\cite{HouWu:97,HouWC:99,EfHou:09} (with energy-error bound of the form $\left(\eps^{\nicefrac12}+H+(\eps/H)^{\nicefrac12}\right)$ in the periodic case), its variant using oversampling~\cite{HouWu:97,EfHoW:00} (with improved error bound of the form $\left(\eps^{\nicefrac12}+H+\eps/H\right)$ in the periodic case), or the Petrov--Galerkin variant of the msFEM using oversampling~\cite{HouWZ:04}. Let us also mention~\cite{AlBri:05} (see also~\cite{HesZZ:14}), which is an extension to arbitrary orders of approximation of the classical msFEM (with error bound of the form $\left(\eps^{\nicefrac12}+H^k+(\eps/H)^{\nicefrac12}\right)$ in the periodic case using $H^1$-conforming finite elements of degree $k\geq 1$).
These methods all rely on the assumption that a conforming finite element basis is available for the (coarse) mesh under consideration. 
Recent research directions essentially focus on the approximation of problems that do not assume scale separation, and on reducing and possibly eliminating the cell resonance error. One can cite, e.g., the Generalized msFEM (GmsFEM)~\cite{EfGaH:13}, or the Local Orthogonal Decomposition (LOD) approach~\cite{HenPe:13,MalPe:14}.
We also mention that other paradigms exist to approximate oscillatory problems, like the Heterogeneous Multiscale Method (HMM)~\cite{EEngq:03,AbEVE:12}. 

On general polytopal meshes, the literature on multiscale methods is more scarce. For constructions in the spirit of the msFEM, one can cite the msFEM \`a la Crouzeix--Raviart of~\cite{LBLLo:13,LBLLo:14}, the so-called Multiscale Hybrid-Mixed (MHM)~\cite{ArHPV:13,PaVaV:17} approach, and the (polynomial-based) method of~\cite{EfLaS:15} in the HDG context. Each one of these methods has its proper design, but they all share the same construction principles: they are based, more or less directly, on oscillatory basis functions that solve local Neumann problems with polynomial boundary data, and result in global systems (posed on the coarse mesh) that can be expressed in terms of face unknowns only. In the following, we will thus refer to those methods as skeletal-msFEM.
The MHM approach actually presents a small difference with respect to the two other approaches since it is based on a hybridized primal formulation, which leads to consider flux-type unknowns at interfaces instead of potential-type unknowns; as a consequence, and in order to impose the compatibility constraint, one needs to solve a saddle-point global problem, whereas for the two other approaches, one ends up with a coercive problem.
For the msFEM \`a la Crouzeix--Raviart, an error bound of the form $\left(\eps^{\nicefrac12}+H+(\eps/H)^{\nicefrac12}\right)$ is proved in~\cite{LBLLo:13} in the periodic case.
However, the analysis is led under the assumption that there exists a finite number of reference elements in the mesh sequence.
For the MHM approach, which is designed in the same spirit, the same type of upper bound for the error is expected. Yet, in~\cite{PaVaV:17}, the authors claim that their method is able to get rid of the resonance error (without oversampling); we clarify this issue in Remark~\ref{re:mhm} below.
For the HDG-like method, the analysis that is provided in~\cite{EfLaS:15} is sharp only in the regime $H\ll\eps$.
As a consequence, there is, to date, no complete polytopal analysis available in the literature for skeletal-msFEM.
Moreover, we observe that in the three methods, the discretization of the (non-oscillatory) right-hand side is realized in a somewhat suboptimal way, which can become a limiting issue in a multi-query context. In the msFEM \`a la Crouzeix--Raviart, the discretization is realized through a projection of the loading term onto the space spanned by the oscillatory basis functions. In the MHM and HDG-like approaches, the whole (local) space $H^1$ is considered. In all cases, the approximation of the right-hand side does not take advantage of the fact that the latter is non-oscillatory.
Let us mention, as another construction in the spirit of the msFEM, the work~\cite{Konat:17}, which exploits in the DG context the ideas introduced in~\cite{AlBri:05}. The drawback, which is inherent to DG methods, is the large size of the online systems.
For constructions in the spirit of the GmsFEM, let us mention in the HDG context the contributions~\cite{EfLMS:15,ChuFY:18} (that are based on~\cite{EfLaS:15}), and the work~\cite{MuWaY:16} in the WG context.

In this work, we devise a multiscale HHO (msHHO) method, which can be seen as a generalization (in particular to arbitrary orders of approximation) of the msFEM \`a la Crouzeix--Raviart of~\cite{LBLLo:13,LBLLo:14}.
Our contribution is twofold.
First, we provide an analysis (in the periodic setting) of the method that is valid on general polytopal mesh sequences (in particular, we do not postulate the existence of reference elements); in that respect, this work presents the first complete polytopal analysis of a skeletal-msFEM.
Note that considering general element shapes in the periodic setting is clearly not a good strategy (cf., e.g.,~\cite{Glori:12}); however, this setting is not our final target.
Second, taking advantage of the flexibility offered by the HHO framework, we improve on the existing methods.
We introduce (polynomial) cell unknowns, that we use for the integration of the right-hand side (cf.~Remarks~\ref{rem:disc_rhs} and~\ref{rem:disc_rhs_2} below). The non-oscillatory loading is hence discretized through a coarse-scale polynomial projection, while the size of the online system remains unchanged since the cell unknowns are locally eliminated in the offline step.
Two versions of the msHHO method are proposed herein, both employing polynomials of arbitrary order $k\ge0$ for the face unknowns. For the mixed-order msHHO method, the cell unknowns are polynomials of order $(k-1)$ (if $k\ge1$), whereas they are polynomials of order $k\ge0$ for the equal-order msHHO method. The mixed-order msHHO method does not require stabilization, whereas a simple stabilization (which avoids computing additional oscillatory basis functions) is introduced in the equal-order case. We prove for both methods an energy-error estimate of the form $\left(\eps^{\nicefrac12}+H^{k+1}+(\eps/H)^{\nicefrac12}\right)\enifed g_k(H)$ in the periodic case.
The analysis of the msHHO method differs from that of the monoscale HHO method since the local fine-scale space does not contain polynomial functions up to order $(k+1)$; in this respect, our key approximation result is Lemma~\ref{lem:app.VkpoT} below.
With respect to~\cite{LBLLo:13}, we also simplify the analysis and weaken the regularity assumptions (cf.~Remark~\ref{re:lbll} below).
Our analysis finally sheds new light on the relationship between the non-computed functions of the local virtual space and the associated local stabilization.
To motivate the design and use of a high-order method, we note, as it was already pointed out in~\cite{AlBri:05}, that the upper bound $g_k(H)$ is minimal for $H_k={\left(\eps^{\nicefrac12}/2(k+1)\right)}^{2/(2k+3)}$, hence as $k\geq 0$ increases, $H_k$ increases whereas $g_k(H_k)$ decreases.
The msHHO method we devise is meant to be a first step in the design of an accurate and computationally effective multiscale approach on general meshes. The next step will be to address the resonance phenomenon and the more realistic setting of no scale separation.

The article is organized as follows. In Sections~\ref{se:cont} and~\ref{se:dis} we introduce, respectively, the continuous and discrete settings.
In Section~\ref{se:app}, we introduce the fine-scale approximation space, exhibiting its (oscillatory) basis functions and studying, locally, its approximation properties. In Section~\ref{se:mshho}, we introduce the two versions of the msHHO method, analyze their stability, and derive energy-error estimates. In Section~\ref{se:num.val}, we present some numerical illustrations in the periodic and locally periodic settings.
Finally, in Appendix~\ref{ap:est} we collect some useful estimates on the first-order two-scale expansion.

\section{Continuous setting} \label{se:cont}

From now on, and in order to lead the analysis, we assume that the diffusion matrix $\mathbb{A}_\eps$ satisfies $\mathbb{A}_\eps({\cdot})=\mathbb{A}({\cdot}/\eps)$ in $\Omega$, where $\mathbb{A}$ is a symmetric and $\Z^d$-periodic matrix field on $\R^d$.
Letting $Q\define{(0,1)}^d$, we define, for $1\leq p\leq +\infty$ and $m\in\N^\star$, the following periodic spaces:
\begin{align*}
  L^p_{\rm per}(Q)&\define\left\{v\in L^p_{\rm loc}(\R^d)\mid v\text{ is }\Z^d\text{-periodic}\right\},\\\quad W^{m,p}_{\rm per}(Q)&\define\left\{v\in W^{m,p}_{\rm loc}(\R^d)\mid v\text{ is }\Z^d\text{-periodic}\right\},
\end{align*}
with the classical conventions that $W^{m,2}_{\rm per}(Q)$ is denoted $H^m_{\rm per}(Q)$ and that the subscript ``loc'' can be omitted for $p=+\infty$.
Letting ${\cal S}_d(\R)$ denote the set of real-valued $d\times d$ symmetric matrices, we also define, for real numbers $0<a\leq b$,
$$
{\cal S}_a^b\define\left\{\mathbb{M}\in{\cal S}_d(\R)\mid\forall\vec{\xi}\in\R^d,\,a|\vec{\xi}|^2\leq\mathbb{M}\vec{\xi}{\cdot}\vec{\xi}\leq b|\vec{\xi}|^2\right\}.
$$
We assume that there exist real numbers $0<\alpha\leq\beta$ such that
\begin{equation} \label{eq:ass}
\mathbb{A}({\cdot})\in{\cal S}_\alpha^\beta\text{ a.e.~in }\R^d.
\end{equation}
Assumption~\eqref{eq:ass} ensures that $\mathbb{A}_\eps\in L^\infty(\Omega;\R^{d\times d})$ is such that $\mathbb{A}_\eps({\cdot})\in{\cal S}_\alpha^\beta$ a.e.~in $\Omega$ for any $\eps>0$, and hence guarantees the existence and uniqueness of the solution to~\eqref{eq:osc} in $\Hoz$ for any $\eps>0$.
More importantly, the assumption~\eqref{eq:ass} ensures that the (whole) family ${(\mathbb{A}_\eps)}_{\eps>0}$ {\rm G}-converges~\cite[Section 1.3.2]{Allai:02} to some constant symmetric matrix $\mathbb{A}_0\in{\cal S}_\alpha^\beta$.
Henceforth, we denote $\rho\define\beta/\alpha\geq 1$ the (global) heterogeneity/anisotropy ratio of both ${(\mathbb{A}_\eps)}_{\eps>0}$ and $\mathbb{A}_0$.
Letting $(\vec{e}_1,\ldots,\vec{e}_d)$ denote the canonical basis of $\R^d$, the expression of $\mathbb{A}_0$ is known to read, for integers $1\leq i,j\leq d$,
\begin{equation} \label{eq:Azer}
  {[\mathbb{A}_0]}_{ij}=\int_Q\mathbb{A}\,(\vec{e}_j+\grad \mu_{j}){\cdot}(\vec{e}_i+\grad \mu_{i})=\int_Q\mathbb{A}\,(\vec{e}_j+\grad \mu_{j}){\cdot}\vec{e}_i,
\end{equation}
where, for any integer $1\leq l\leq d$, the so-called corrector $\mu_{l}\in\HoPer$ is the solution with zero mean-value on $Q$ to the problem
\begin{equation} \label{eq:corr}
  \left\{
  \begin{alignedat}{1}
    &-\div(\mathbb{A}(\grad \mu_{l}+\vec{e}_l))=0\quad\text{in $\R^d$}, \\
    &\;\,\mu_{l}\text{ is $\Z^d$-periodic}.
  \end{alignedat}
  \right.
\end{equation}
For further use, we also define the linear operator $\Reps:L^p_{\rm per}(Q)\to L^p(\Omega)$, $1\leq p\leq +\infty$, such that, for any function $\chi\in L^p_{\rm per}(Q)$, $\Reps(\chi)\in L^p(\Omega)$ satisfies $\Reps(\chi)({\cdot})=\chi({\cdot}/\eps)$ in $\Omega$.
In particular, for any integers $1\leq i,j\leq d$, we have ${\left[\mathbb{A}_\eps\right]}_{ij}=\Reps(\mathbb{A}_{ij})$.
A useful property of $\Reps$ is the relation $\partial_l(\Reps(\chi))=\frac{1}{\eps}\Reps(\partial_l\chi)$, valid for any function $\chi\in W^{1,p}_{\rm per}(Q)$ and any integer $1\leq l\leq d$.

The homogenized problem reads
\begin{equation} \label{eq:hom}
  \left\{
  \begin{alignedat}{3}
    -\div(\mathbb{A}_0\grad u_0)&=f&\quad&\text{in $\Omega$}, \\
    u_0&=0&\quad&\text{on $\partial\Omega$}.
  \end{alignedat}
  \right.
\end{equation}
We introduce the so-called first-order two-scale expansion 
\begin{equation} \label{eq:correcteur}
\Leps(u_0)\define u_0+\eps\sum_{l=1}^d\Reps(\mu_l)\partial_lu_0.
\end{equation}
Note that $(u_\eps-\Leps(u_0))$ does not a priori vanish on the boundary of $\Omega$.

\section{Discrete setting} \label{se:dis}

We denote by ${\cal H}\subset\mathbb{R}_+^\star$ a countable set of meshsizes having $0$ as its unique accumulation point, and we consider mesh sequences of the form ${(\TH)}_{H\in{\cal H}}$.
For any $H\in{\cal H}$, a {\em mesh} $\TH$ is a finite collection of nonempty disjoint open polytopes (polygons/polyhedra) $T$, called {\em elements} or {\em cells}, such that $\overline{\Omega}=\bigcup_{T\in\TH}\overline{T}$ and $H=\max_{T\in\TH}H_T$, $H_T$ standing for the diameter of the cell $T$. The mesh cells being polytopal, their boundary is composed of a finite union of portions of affine hyperplanes in $\R^d$ called facets (each facet has positive $(d{-}1)$-dimensional measure). A closed subset $F$ of $\overline{\Omega}$ is called a {\em face} if either
\begin{inparaenum}[(i)]
  \item there exist $T_1,T_2\in\TH$ such that $F=\partial T_1\cap\partial T_2\cap Z$ where $Z$ is an affine hyperplane supporting a facet of both $T_1$ and $T_2$ (and $F$ is termed {\em interface}), or 
  \item there exists $T\in\TH$ such that $F=\partial T\cap\partial\Omega\cap Z$ where $Z$ is an affine hyperplane supporting a facet of both $T$ and $\Omega$ (and $F$ is termed {\em boundary face}).
\end{inparaenum}
Interfaces are collected in the set $\FHi$, boundary faces in $\FHb$, and we let $\FH\define\FHi\cup\FHb$.
The diameter of a face $F\in\FH$ is denoted $H_F$.
For all $T\in\TH$, we define $\FT\define\{F\in\FH\mid F\subset\partial T\}$ the set of faces lying on the boundary of $T$; note that the faces in $\FT$ compose the boundary of $T$.
For any $T\in\TH$, we denote by $\vec{n}_{\partial T}$ the unit normal vector to $\partial T$ pointing outward $T$, and for any $F\in\FT$, we let $\vec{n}_{T,F}\define\vec{n}_{\partial T\mid F}$ (by definition, $\vec{n}_{T,F}$ is a constant vector on $F$).

We adopt the following notion of admissible mesh sequence; cf.~\cite[Section~1.4]{DPErn:12} and~\cite[Definition~2.1]{DPELe:16}.
\begin{definition}[Admissible mesh sequence] \label{def:adm}
The mesh sequence ${(\TH)}_{H\in{\cal H}}$ is {\em admissible} if, for all $H\in{\cal H}$, $\TH$ admits a matching simplicial sub-mesh $\mathfrak{T}_H$ (meaning that the cells in $\mathfrak{T}_H$ are sub-cells of the cells in $\TH$ and that the faces of these sub-cells belonging to the skeleton of $\TH$ are sub-faces of the faces in $\FH$), and there exists a real number $\gamma>0$, called {\em mesh regularity parameter}, such that, for all $H\in{\cal H}$, the following holds:
  \begin{enumerate}[(i)]
    \item For all simplex $S\in\mathfrak{T}_H$ of diameter $H_S$ and inradius $R_S$, $\gamma H_S\leq R_S$;
    \item For all $T\in\TH$, and all $S\in\mathfrak{T}_T\define\{S\in\mathfrak{T}_H\mid S\subseteq T\}$, $\gamma H_T\leq H_S$.
  \end{enumerate}
\end{definition}
\noindent
Two classical consequences of Definition~\ref{def:adm} are that, for any mesh $\TH$ belonging to an admissible mesh sequence, \begin{inparaenum}[(i)] \item the quantity ${\rm card}(\FT)$ is bounded independently of the diameter $H_T$ for all $T\in\TH$~\cite[Lemma 1.41]{DPErn:12}, and \item mesh faces have a comparable diameter to the diameter of the cells to which they belong~\cite[Lemma 1.42]{DPErn:12}\end{inparaenum}.

For any $q\in\N$, and any integer $1\leq l\leq d$, we denote by $\mathbb{P}^q_l$ the linear space spanned by $l$-variate polynomial functions of total degree less or equal to $q$.
We let
$${\rm N}^q_l\define{\rm dim}(\mathbb{P}^q_l)=\begin{pmatrix}q+l\\q\end{pmatrix}.$$
Let a mesh $\TH$ be given. For any $T\in\TH$, $\mathbb{P}^q_d(T)$ is composed of the restriction to $T$ of polynomials in $\mathbb{P}^q_d$, and for any $F\in\FH$, $\mathbb{P}^q_{d-1}(F)$ is composed of the restriction to $F$ of polynomials in $\mathbb{P}^q_d$ (this space can also be described as
the restriction to $F$ of polynomials in $\mathbb{P}^q_{d{-}1}\circ\Theta^{-1}$, where $\Theta$ is any affine bijective mapping from $\R^{d{-}1}$ to the affine hyperplane supporting $F$).
We also introduce, for any $T\in\TH$, the following broken polynomial space:
$$\mathbb{P}^q_{d-1}(\FT)\define\left\{v\in\L[\partial T]\mid v_{\mid F}\in\mathbb{P}^q_{d-1}(F)\;\forall F\in\FT\right\}.$$
The term `broken' refers to the fact that no continuity is required between adjacent faces for functions in $\mathbb{P}^q_{d-1}(\FT)$.
For any $T\in\TH$, we denote by ${(\Phi_T^{q,i})}_{1\leq i\leq {\rm N}^q_d}$ a set of basis functions of the space $\mathbb{P}^q_d(T)$, and for any $F\in\FH$, we denote by ${(\Phi_F^{q,j})}_{1\leq j\leq {\rm N}^q_{d{-}1}}$ a set of basis functions of the space $\mathbb{P}^q_{d-1}(F)$.
We define, for any $T\in\TH$ and $F\in\FH$, $\Pi^q_T$ and $\Pi^q_F$ as the $L^2$-orthogonal projectors onto $\mathbb{P}^q_d(T)$ and $\mathbb{P}^q_{d-1}(F)$, respectively. Whenever no confusion can arise, we write, for all $T\in\TH$, all $F\in\FT$, and all $v\in H^1(T)$, $\Pi^q_F(v)$ instead of $\Pi^q_F(v_{\mid F})$.

We conclude this section by recalling some classical results, that are valid for any mesh $\TH$ belonging to an admissible mesh sequence in the sense of Definition~\ref{def:adm}. For any $T\in\TH$ and $F\in\FT$, the trace inequalities
\begin{align}
  \norm{\L[F]}{v}&\leq c_{\rm tr,d}H_F^{-\nicefrac12}\norm{\L[T]}{v}&\forall&v\in\mathbb{P}^q_d(T), \label{eq:tr.d} \\
  \norm{\L[F]}{v}&\leq c_{\rm tr,c}{\left(H_T^{-1}\norm{\L[T]}{v}^2+H_T\norm{\L[T]^d}{\grad v}^2\right)}^{\nicefrac12}&\forall& v\in H^1(T), \label{eq:tr.c}
\end{align}
hold~\cite[Lemmas 1.46 and 1.49]{DPErn:12}, as well as the local Poincar\'e inequality
\begin{equation} \label{eq:poin}
  \norm{\L[T]}{v}\leq c_{\rm P}H_T\norm{\L[T]^d}{\grad v}\quad\forall v\in H^1(T)\text{ such that }\int_T v=0,
\end{equation}
where $c_{\rm P}=\pi^{-1}$ for convex elements~\cite{Beben:03}; estimates in the non-convex case can be found, e.g., in~\cite{VeeVe:12}. Finally, proceeding as in~\cite[Lemma~5.6]{ErnGu:17}, one can prove using the above trace and Poincar\'e inequalities that 
\begin{equation} \label{eq:app}
  \senorm{H^m(T)}{v-\Pi^q_T(v)}+H_T^{\nicefrac12}\senorm{H^m(F)}{v-\Pi^q_T(v)}\leq c_{\rm app}H_T^{s-m}\senorm{H^s(T)}{v}\quad\forall v\in H^s(T),
\end{equation}
for integers $1\leq s\leq q+1$ and $0\leq m\leq (s-1)$. All of the above constants are independent of the meshsize and can only depend on the underlying polynomial degree $q$, the space dimension $d$, and the mesh regularity parameter $\gamma$. 

Henceforth, we use the symbol $c$ to denote a generic positive constant, whose value can change at each occurrence, provided it is independent of the micro-scale $\eps$, any meshsize $H_T$ or $H$, and the homogenized solution $u_0$. We also track the direct dependency of the error bounds on the parameters $\alpha,\beta$ characterizing the spectrum of the diffusion matrix. The value of the generic constant $c$ can depend on the space dimension $d$, the underlying polynomial degree, the mesh regularity parameter $\gamma$, and on some higher-order norms of the rescaling $\mathbb{A}/\beta$ of the diffusion matrix or the correctors $\mu_{l}$ that will be made clear from the context. 

\section{Fine-scale approximation space} \label{se:app}

Let $k\in\N$ and let $\TH$ be a member of an admissible mesh sequence in the sense of Definition~\ref{def:adm}.
In this section, we introduce the fine-scale approximation space on which we will base our multiscale HHO method. We first construct in Section~\ref{sse:bas} a set of cell-based and face-based basis functions, then we provide in Section~\ref{sse:char} a local characterization of the underlying space, finally we study its approximation properties in Section~\ref{sse:app}.

\subsection{Oscillatory basis functions} \label{sse:bas}

The oscillatory basis functions consist of cell- and face-based basis functions.

\subsubsection{Cell-based basis functions} \label{ssse:bas}

Let $T\in\TH$.
If $k=0$, we do not define cell-based basis functions. Assume now that $k\geq 1$.
For all $1\leq i\leq {\rm N}^{k-1}_d$, we consider the problem
\begin{equation} \label{eq:bas.cell.op}
  \inf\left\{\int_T\left[\frac{1}{2}\mathbb{A}_\eps\grad\varphi{\cdot}\grad\varphi-\Phi_T^{k-1,i}\,\varphi\right],\,\varphi\in H^1(T),\,\PikF(\varphi)=0\;\forall F\in\FT\right\}.
\end{equation}
Problem~\eqref{eq:bas.cell.op} admits a unique minimizer.
This minimizer, that we will denote $\varphi_{\eps,T}^{k+1,i}\in H^1(T)$, can be proved to solve, for real numbers ${(\lambda_{F,j}^T)}_{F\in\FT,\,1\leq j\leq {\rm N}^k_{d{-}1}}$ satisfying the compatibility condition
\begin{equation*} \label{eq:comp.cell}
  \sum_{F\in\FT}\int_F\sum_{j=1}^{{\rm N}^k_{d{-}1}}\lambda_{F,j}^T\Phi_F^{k,j}=-\int_T\Phi_T^{k-1,i},
\end{equation*}
the constrained Neumann problem
\begin{equation} \label{eq:bas.cell.pr}
  \left\{
  \begin{alignedat}{3}
    -\div(\mathbb{A}_\eps\grad\varphi_{\eps,T}^{k+1,i})&=\Phi_T^{k-1,i}&\qquad&\text{in $T$}, \\
    \mathbb{A}_\eps\grad\varphi_{\eps,T}^{k+1,i}{\cdot}\vec{n}_{T,F}&=\sum_{j=1}^{{\rm N}^k_{d{-}1}}\lambda_{F,j}^T\Phi_F^{k,j}&\qquad&\text{on all $F\in\FT$}, \\
    \PikF(\varphi_{\eps,T}^{k+1,i})&=0&\qquad&\text{on all $F\in\FT$}.
  \end{alignedat}
  \right.
\end{equation}
The superscript $k+1$ is meant to remind us that the functions $\varphi_{\eps,T}^{k+1,i}$ are used to generate a linear space which has the same approximation capacity as the polynomial space of order at most $k+1$, as will be shown in Section~\ref{sse:app}. 

\begin{remark}[Practical computation] \label{re:prac}
To compute $\varphi_{\eps,T}^{k+1,i}$ for all $1\leq i\leq {\rm N}^{k-1}_d$, one considers in practice a (shape-regular) matching simplicial mesh ${\cal T}^T_h$ of the cell $T$, with size $h$ smaller than $\eps$. Then, one can solve Problem~\eqref{eq:bas.cell.pr} approximately by using a classical (equal-order) monoscale HHO method (or any other monoscale approximation method). For the implementation of the monoscale HHO method, we refer to \cite{CiDPE:17}.
One can either consider a weak formulation in $\left\{\varphi\in H^1(T),\,\PikF(\varphi)=0\,\forall F\in\FT\right\}$, which leads to a coercive problem, or a weak formulation in $H^1(T)$, which leads to a saddle-point system with Lagrange multipliers. Equivalent considerations apply below to the computation of the face-based basis functions.
Note that the error estimates we provide in this work for our approach do not take into account the local approximations of size $h$ and assume that~\eqref{eq:bas.cell.pr} and~\eqref{eq:bas.face.pr} below are solved exactly.
\end{remark}

\subsubsection{Face-based basis functions}

Let $T\in\TH$. For all $F\in\FT$ and all $1\leq j\leq {\rm N}^k_{d{-}1}$, we consider the problem
\begin{equation} \label{eq:bas.face.op}
  \inf\left\{\int_{T}\left[\frac{1}{2}\mathbb{A}_\eps\grad\varphi{\cdot}\grad\varphi\right],\,\varphi\in H^1(T),\,\PikF(\varphi)=\Phi_F^{k,j},\,\Pi^k_\sigma(\varphi)=0\;\forall\sigma\in{\cal F}_{T}\setminus\{F\}\right\}.
\end{equation}
Problem~\eqref{eq:bas.face.op} admits a unique minimizer.
This minimizer, that we will denote $\varphi_{\eps,T,F}^{k+1,j}\in H^1(T)$, can be proved to solve, for real numbers ${(\lambda_{\sigma,q}^{T,F})}_{\sigma\in{\cal F}_{T},\,1\leq q\leq {\rm N}^k_{d{-}1}}$ satisfying the compatibility condition
\begin{equation*} \label{eq:comp.face}
  \sum_{\sigma\in{\cal F}_{T}}\int_\sigma\sum_{q=1}^{{\rm N}^k_{d{-}1}}\lambda_{\sigma,q}^{T,F}\Phi_\sigma^{k,q}=0,
\end{equation*}
the constrained Neumann problem
\begin{equation} \label{eq:bas.face.pr}
  \left\{
  \begin{alignedat}{3}
    -\div(\mathbb{A}_\eps\grad\varphi_{\eps,T,F}^{k+1,j})&=0&\qquad&\text{in $T$}, \\
    \mathbb{A}_\eps\grad\varphi_{\eps,T,F}^{k+1,j}{\cdot}\vec{n}_{T,\sigma}&=\sum_{q=1}^{{\rm N}^k_{d{-}1}}\lambda_{\sigma,q}^{T,F}\Phi_\sigma^{k,q}&\qquad&\text{on all $\sigma\in{\cal F}_{T}$}, \\
    \PikF(\varphi_{\eps,T,F}^{k+1,j})&=\Phi_F^{k,j},&\qquad&\text{on $F$}, \\
    \Pi^k_\sigma(\varphi_{\eps,T,F}^{k+1,j})&=0&\qquad&\text{on all $\sigma\in{\cal F}_{T}\setminus\{F\}$}.
  \end{alignedat}
  \right.
\end{equation}

\subsection{Discrete space} \label{sse:char}

We introduce, for any $T\in\TH$, the space
\begin{equation} \label{eq:VkpoT}
  \VkpoT\define\left\{v_\eps\in H^1(T)\mid\div(\mathbb{A}_\eps\grad v_\eps)\in\mathbb{P}^{k-1}_d(T),\,\mathbb{A}_\eps\grad v_\eps{\cdot}\vec{n}_{\partial T}\in\Pkdmo{\FT}\right\},
\end{equation}
with the convention that $\mathbb{P}^{-1}_d(T)\define\{0\}$.
We recall that the condition $\mathbb{A}_\eps\grad v_\eps{\cdot}\vec{n}_{\partial T}\in\Pkdmo{\FT}$ is equivalent to $\mathbb{A}_\eps\grad v_\eps{\cdot}\vec{n}_{T,F}\in\Pkdmo{F}$ for all $F\in\FT$.
Proceeding as in~\cite[Section~2.4]{CoDPE:16}, it can be easily shown that the dimension of $\VkpoT$ is $\left({\rm N}^{k-1}_d+{\rm card}(\FT)\times{\rm N}^k_{d{-}1}\right)$ (or ${\rm card}(\FT)$ if $k=0$).
\begin{proposition}[Characterization of $\VkpoT$] \label{pr:span}
For any $T\in\TH$, the family
$$\left\{{(\varphi_{\eps,T}^{k+1,i})}_{1\leq i\leq{\rm N}^{k-1}_d},{(\varphi_{\eps,T,F}^{k+1,j})}_{F\in\FT,\,1\leq j\leq{\rm N}^k_{d{-}1}}\right\}$$
forms a basis for the space $\VkpoT$.
\end{proposition}
\begin{proof}
To establish the result, we only need to prove that
$$
\VkpoT\subset{\rm Span}\left\{{(\varphi_{\eps,T}^{k+1,i})}_{1\leq i\leq{\rm N}^{k-1}_d},{(\varphi_{\eps,T,F}^{k+1,j})}_{F\in\FT,\,1\leq j\leq{\rm N}^k_{d{-}1}}\right\},
$$
since the converse inclusion follows from the definition of the oscillatory basis functions, and the cardinal of the family fits the dimension of $\VkpoT$.
Let $v_\eps\in\VkpoT$. Then, there exist real numbers ${(\theta_T^i)}_{1\leq i\leq{\rm N}^{k-1}_d}$ (only if $k\geq 1$) and ${(\theta_{T,F}^j)}_{F\in\FT,\,1\leq j\leq{\rm N}^k_{d{-}1}}$, satisfying the compatibility condition
  \begin{equation*} 
    \sum_{F\in\FT}\int_F\sum_{j=1}^{{\rm N}^k_{d{-}1}}\theta_{T,F}^j\Phi_F^{k,j}=-\int_T\sum_{i=1}^{{\rm N}^{k-1}_d}\theta_T^i\Phi_T^{k-1,i} (=0\text{ if }k=0),
  \end{equation*}
  such that
  \begin{equation*} 
    \left\{
    \begin{alignedat}{3}
      -\div(\mathbb{A}_\eps\grad v_\eps)&=\sum_{i=1}^{{\rm N}^{k-1}_d}\theta_T^i\Phi_T^{k-1,i} (=0\text{ if }k=0)&\qquad&\text{in $T$}, \\
      \mathbb{A}_\eps\grad v_\eps{\cdot}\vec{n}_{T,F}&=\sum_{j=1}^{{\rm N}^k_{d{-}1}}\theta_{T,F}^j\Phi_F^{k,j}&\qquad&\text{on all $F\in\FT$}.
    \end{alignedat}
    \right.
  \end{equation*}
  Let us now introduce
$$
\zeta\define v_\eps-\sum_{i=1}^{{\rm N}^{k-1}_d}\theta_T^i\varphi_{\eps,T}^{k+1,i}-\sum_{\sigma\in\FT}\sum_{j=1}^{{\rm N}^k_{d{-}1}}x^{k,j}_{\sigma}(v_\eps)\varphi_{\eps,T,\sigma}^{k+1,j},
$$
  where, for all $\sigma\in \FT$, the real numbers ${\left(x_{\sigma}^{k,j}(v_\eps)\right)}_{1\leq j\leq{\rm N}^k_{d{-}1}}$ solve the linear system
  $$\sum_{j=1}^{{\rm N}^k_{d{-}1}}\left(\int_{\sigma} \Phi_{\sigma}^{k,j}\,\Phi_{\sigma}^{k,q}\right)x_{\sigma}^{k,j}(v_\eps)=\int_{\sigma} v_\eps\,\Phi_{\sigma}^{k,q}\qquad\text{for all $1\leq q\leq {\rm N}^k_{d{-}1}$}.
$$
It can be easily checked that $-\div(\mathbb{A}_\eps\grad\zeta)=0$ in $T$ and that
$\mathbb{A}_\eps\grad\zeta{\cdot}\vec{n}_{T,F}\in\Pkdmo{F}$ and  $\PikF(\zeta)=0$
on all $F\in\FT$. Using the compatibility conditions, we also infer that 
$\int_{\partial T}\mathbb{A}_\eps\grad\zeta{\cdot}\vec{n}_{\partial T}=0$,
which means that the previous system for $\zeta$ is compatible.
Hence, $\zeta\equiv 0$, which concludes the proof.
\end{proof}
\begin{remark}[Space $\VkpoT$]
The definition of the space $\VkpoT$ is reminiscent of that considered in the 
non-conforming VEM in the case where $\mathbb{A}_\eps=\mathbb{I}_d$; see~\cite{AyuLM:16}
and also~\cite{CoDPE:16}.
\end{remark}
We define $H_{\partial T}\in\mathbb{P}^0_{d-1}(\FT)$ such that, for any $F\in\FT$, $H_{\partial T\mid F}\define H_F$.
We will need the following inverse inequality on the normal component of $\mathbb{A}_\eps\grad v_\eps$ for a function $v_\eps\in \VkpoT$; for completeness, we also establish a bound on the divergence.
\begin{lemma}[Inverse inequalities] \label{lem:inv.VkpoT}
The following holds for all $v_\eps\in\VkpoT$:
\begin{equation}
H_T\norm{L^2(T)}{\div(\mathbb{A}_\eps\grad v_\eps)}
+ \norm{L^2(\partial T)}{H_{\partial T}^{\nicefrac12}\mathbb{A}_\eps\grad v_\eps{\cdot}\vec{n}_{\partial T}}
\le c\, \beta^{\nicefrac12} \norm{L^2(T)^d}{\mathbb{A}_\eps^{\nicefrac12}\grad v_\eps},
\end{equation}
with $c$ independent of $\eps$, $H_T$, $\alpha$ and $\beta$.
\end{lemma}
\begin{proof}
Note that the functions on the left-hand side are (piecewise) polynomials, but the function on the right-hand side is not a polynomial in general. Let us first bound the divergence. Let $d_\eps\define \div(\mathbb{A}_\eps\grad v_\eps)\in \mathbb{P}^{k-1}_d(T)$. Let $S$ be a simplicial sub-cell of $T$. Considering the standard bubble function $b_S\in H^1_0(S)$ (equal to the scaled product of the barycentric coordinates in $S$ taking the value one at the barycenter of $S$), we infer using integration by parts that, for some $c>0$ depending on mesh regularity,
\begin{align*}
c\, \norm{L^2(S)}{d_\eps}^2 &\leq \int_S d_\eps b_S d_\eps = \int_S \div(\mathbb{A}_\eps\grad v_\eps)b_S d_\eps \\&= -\int_S \mathbb{A}_\eps\grad v_\eps{\cdot}\grad(b_S d_\eps)
\leq \beta^{\nicefrac12} \norm{L^2(S)^d}{\mathbb{A}_\eps^{\nicefrac12}\grad v_\eps}H_S^{-1}\norm{L^2(S)}{d_\eps},
\end{align*}
where the last bound follows by applying an inverse inequality to the polynomial function $b_S d_\eps$. Summing over all the simplicial sub-cells and invoking mesh regularity, we conclude that $\norm{L^2(T)}{\div(\mathbb{A}_\eps\grad v_\eps)} \le c\, \beta^{\nicefrac12}H_T^{-1} \norm{L^2(T)^d}{\mathbb{A}_\eps^{\nicefrac12}\grad v_\eps}$. 
Let us now bound the normal component at the boundary. Let $\sigma$ be a sub-face of a face $F\in\FT$, and let $S\subseteq T$ be the simplex of the sub-mesh such that $\sigma$ is a face of $S$. Then, $r_S\define [\div(\mathbb{A}_\eps\grad v_\eps)]_{\mid S} \in \mathbb{P}^{k-1}_d(S)\subset \mathbb{P}^{k}_d(S)$ and $r_\sigma\define [\mathbb{A}_\eps\grad v_\eps{\cdot}\vec{n}_{\partial T}]_{\mid\sigma} \in \mathbb{P}^{k}_{d-1}(\sigma)$. Note that $\vec{n}_{\partial T\mid\sigma}=\vec{n}_{\partial S\mid\sigma}$. Invoking~\cite[Lemma~A.3]{ErnVo:16}, we infer that there is a vector-valued polynomial function $\vec{q}$ in the Raviart--Thomas--N\'ed\'elec (RTN) finite element space of order $k$ in $S$ so that $\div(\vec{q})=r_S$ in $S$, $\vec{q}{\cdot}\vec{n}_{\partial T\mid\sigma}=r_\sigma$ on $\sigma$, and
\[
\|\vec{q}\|_{L^2(S)^d} \le c' \min_{\substack{\vec{z}\in \vec{H}(\mathrm{div};S)\\\div (\vec{z}) = r_S\text{ in $S$}\\\vec{z}{\cdot}\vec{n}_{\partial T\mid\sigma}=r_\sigma\text{ on $\sigma$}}} \|\vec{z}\|_{L^2(S)^d},
\]
with $c'$ depending on $\gamma$ (but not on $k$) and $\vec{H}(\mathrm{div};S)\define\{\vec{z}\in L^2(S)^d \mid \div(\vec{z})\in L^2(S)\}$. Since the function $[\mathbb{A}_\eps\grad v_\eps]_{\mid S}$ is in $\vec{H}(\mathrm{div};S)$ and satisfies the requested conditions on the divergence in $S$ and the normal component on $\sigma$, we conclude that $\|\vec{q}\|_{L^2(S)^d} \le c'\|\mathbb{A}_\eps\grad v_\eps\|_{L^2(S)^d}$. A discrete trace inequality in the RTN finite element space shows that 
\[
\norm{L^2(\sigma)}{\mathbb{A}_\eps\grad v_\eps{\cdot}\vec{n}_{\partial T}}
= \norm{L^2(\sigma)}{\vec{q}{\cdot}\vec{n}_{\partial T}}
\le c \,H_\sigma^{-\nicefrac12} \norm{L^2(S)^d}{\vec{q}} \le 
c \,H_\sigma^{-\nicefrac12} \norm{L^2(S)^d}{\mathbb{A}_\eps\grad v_\eps},
\]
where $c$ depends on $\gamma$ and $k$. We conclude by invoking mesh regularity.
\end{proof}

\subsection{Approximation properties} \label{sse:app}

We now investigate the approximation properties of the space $\VkpoT$, for all $T\in\TH$. Our aim is to study how well the first-order two-scale expansion $\Leps(u_0)$ can be approximated in the discrete space $\VkpoT$. 
Let us define $\pi_{\eps,T}^{k+1}(u_0)\in\VkpoT$ such that $\displaystyle\int_T\pi_{\eps,T}^{k+1}(u_0)=\int_T\Leps(u_0)$ and
\begin{equation} \label{eq:proj}
  \left\{
  \begin{alignedat}{3}
    -\div(\mathbb{A}_\eps\grad\pi_{\eps,T}^{k+1}(u_0))&=-\div(\mathbb{A}_0\grad\Pi^{k+1}_T(u_0))\in\mathbb{P}^{k-1}_d(T)&\qquad&\text{in $T$}, \\
    \mathbb{A}_\eps\grad\pi_{\eps,T}^{k+1}(u_0){\cdot}\vec{n}_{\partial T}&=\mathbb{A}_0\grad\Pi^{k+1}_T(u_0){\cdot}\vec{n}_{\partial T}\in\Pkdmo{\FT}&\qquad&\text{on $\partial T$}.
  \end{alignedat}
  \right.
\end{equation}
Note that the data in~\eqref{eq:proj} are compatible.
From~\eqref{eq:proj} we infer that, for any $w\in H^1(T)$,
\begin{equation} \label{eq:proj.wf}
  \int_T\mathbb{A}_\eps\grad\pi_{\eps,T}^{k+1}(u_0){\cdot}\grad w=\int_T\mathbb{A}_0\grad\Pi^{k+1}_T(u_0){\cdot}\grad w.
\end{equation}

\begin{lemma}[Approximation in $\VkpoT$] \label{lem:app.VkpoT}
  Assume that the correctors $\mu_{l}$ are in $W^{1,\infty}(\R^d)$ for any $1\leq l\leq d$, and that $u_0\in H^{k+2}(T)\cap W^{1,\infty}(T)$. Then, the following holds: 
\begin{multline} 
    \norm{\L[T]^d}{\mathbb{A}_\eps^{\nicefrac12}\grad(\Leps(u_0)-\pi_{\eps,T}^{k+1}(u_0))}\leq c\,\beta^{\nicefrac12}\rho^{\nicefrac12}\bigg(H_T^{k{+}1}\senorm{H^{k+2}(T)}{u_0} \\+ \eps \senorm{H^2(T)}{u_0} + \eps^{\nicefrac12} |\partial T|^{\nicefrac12}\senorm{W^{1,\infty}(T)}{u_0}\bigg),
    \label{eq:app.VkpoT}
\end{multline}
  with $c$ independent of $\eps$, $H_T$, $u_0$, $\alpha$, $\beta$, and possibly depending on $d$, $k$, $\gamma$, $\displaystyle\max_{1\leq l\leq d}\norm{W^{1,\infty}(\R^d)}{\mu_l}$.
\end{lemma}

\begin{proof}
Subtracting/adding $\mathbb{A}_0\grad u_0$ and using~\eqref{eq:proj.wf} with $w=\Leps(u_0)_{|T}-\pi_{\eps,T}^{k+1}(u_0)$ which is in $H^1(T)$, we infer that 
\begin{multline*}
    \norm{\L[T]^d}{\mathbb{A}_\eps^{\nicefrac12}\grad (\Leps(u_0)-\pi_{\eps,T}^{k+1}(u_0))}^2=\int_T(\mathbb{A}_\eps\grad \Leps(u_0)-\mathbb{A}_0\grad u_0){\cdot}\\\grad (\Leps(u_0)-\pi_{\eps,T}^{k+1}(u_0))+\int_T\mathbb{A}_0\grad(u_0-\Pi^{k+1}_T(u_0)){\cdot}\grad(\Leps(u_0)-\pi_{\eps,T}^{k+1}(u_0)).
\end{multline*}
Using the Cauchy--Schwarz inequality and the fact that $\Leps(u_0)_{\mid T}-\pi_{\eps,T}^{k+1}(u_0)$ has zero mean-value on $T$ by construction, we infer that
\begin{align*} 
    \norm{\L[T]^d}{\mathbb{A}_\eps^{\nicefrac12}\grad (\Leps(u_0)-\pi_{\eps,T}^{k+1}(u_0))}\leq {}&\beta^{\nicefrac12}\rho^{\nicefrac12}\norm{L^2(T)^d}{\grad(u_0-\Pi^{k+1}_T(u_0))}\\&+\alpha^{-\nicefrac12}\sup_{w\in H^1_\star(T)}\frac{\left|\mathcal{F}_\eps(w)\right|}{\norm{\L[T]^d}{\grad w}},
\end{align*}
with $\mathcal{F}_\eps(w) = \int_T(\mathbb{A}_\eps\grad\Leps(u_0)-\mathbb{A}_0\grad u_0){\cdot}\grad w$ and $H^1_\star(T)=\{w\in H^1(T) \mid \int_Tw=0\}$. The first term in the right-hand side is bounded using the approximation properties~\eqref{eq:app} of $\Pi^{k+1}_T$ with $m=1$ and $s=k+2$, and the second term is bounded in Lemma~\ref{lem:dual_Neu1} (take $D=T$).
\end{proof}

\begin{remark}[Alternative estimate] \label{re:lbll}
An alternative estimate to~\eqref{eq:app.VkpoT} can be derived under the slightly stronger regularity assumptions
that there is $\kappa>0$ so that $\mathbb{A}\in C^{0,\kappa}(\R^d;\R^{d\times d})$, and that $u_0\in H^{\max(k+2,3)}(T)$.
The proof of this estimate follows the strategy advocated 
in~\cite{LBLLo:13}, where 
one invokes Lemma~\ref{lem:dual_Neu2} instead of Lemma~\ref{lem:dual_Neu1} at the end of the proof of Lemma~\ref{lem:app.VkpoT} to infer that
  \begin{multline*} 
    \norm{\L[T]^d}{\mathbb{A}_\eps^{\nicefrac12}\grad(\Leps(u_0)-\pi_{\eps,T}^{k+1}(u_0))}\leq c\,\beta^{\nicefrac12}\rho^{\nicefrac12}\bigg(H_T^{k{+}1}\senorm{H^{k+2}(T)}{u_0} \\ + \left(\eps+(\eps H_T)^{\nicefrac12}\right) \senorm{H^2(T)}{u_0} + \eps H_T\senorm{H^3(T)}{u_0} + \eps^{\nicefrac12} H_T^{-\nicefrac12}\senorm{H^1(T)}{u_0}\bigg),
  \end{multline*}
with $c$ independent of $\eps$, $H_T$, $u_0$, $\alpha$, $\beta$, and possibly depending on $d$, $k$, $\gamma$, $\norm{C^{0,\kappa}(\R^d;\R^{d\times d})}{\mathbb{A}/\beta}$.
This local estimate leads to the same global error estimate for (both versions of) the msHHO method described hereafter than~\eqref{eq:app.VkpoT}; see in particular the end of the proof of Theorem~\ref{th:est.err}.
\end{remark}

\section{The msHHO method} \label{se:mshho}

In this section, we introduce and analyze the multiscale HHO (msHHO) method. We consider first in Section~\ref{sec:mixed_order} a mixed-order version and then in Section~\ref{sec:equal_order} an equal-order version concerning the polynomial degree used for the cell and face unknowns.
Let $\TH$ be a member of an admissible mesh sequence in the sense of Definition~\ref{def:adm}.

\subsection{The mixed-order case}
\label{sec:mixed_order}

Let $k\geq1$. For all $T\in\TH$, we consider the following local set of discrete unknowns:
\begin{equation} \label{eq:UkT}
  \UkT\define\Pkmd{T}\times\Pkdmo{\FT}.
\end{equation}
Any element $\uvT\in\UkT$ is decomposed as $\uvT\define(\vT,\vpT)$.
For any $F\in\FT$, we denote $\vF\define{\rm v}_{\FT\mid F}\in\Pkdmo{F}$.
We introduce the local reduction operator $\IkT:H^1(T)\to\UkT$ such that, for any $v\in H^1(T)$, $\IkT v\define(\Pi^{k-1}_T(v),\Pi^k_{\partial T}(v))$, where $\Pi^k_{\partial T}(v)\in\Pkdmo{\FT}$ is defined, for any $F\in\FT$, by ${\Pi^k_{\partial T}(v)}_{\mid F}\define\PikF(v)$.
Reasoning as in \cite[Section 2.4]{CoDPE:16}, 
it can be proved that, for all $T\in\TH$, the restriction of $\IkT$ to $\VkpoT$ is an isomorphism from $\VkpoT$ to $\UkT$. Thus,
the triple $(T,\VkpoT,\IkT)$ defines a finite element in the sense of Ciarlet.

We define the local multiscale reconstruction operator $p_{\eps,T}^{k+1}:\UkT\to\VkpoT$ such that, for any $\uvT=(\vT,\vpT)\in\UkT$, $p_{\eps,T}^{k+1}(\uvT)\in\VkpoT$ satisfies $\displaystyle\int_T p_{\eps,T}^{k+1}(\uvT)=\int_T\vT$ and solves, for all $w_\eps\in\VkpoT$, the well-posed local Neumann problem
\begin{equation} \label{eq:rec.op}
  \int_T\mathbb{A}_\eps\grad p_{\eps,T}^{k+1}(\uvT){\cdot}\grad w_\eps=-\int_T\vT\,\div(\mathbb{A}_\eps\grad w_\eps)+\int_{\partial T}\vpT\,\mathbb{A}_\eps\grad w_\eps{\cdot}\vec{n}_{\partial T}.
\end{equation}
Note that~\eqref{eq:rec.op} can be equivalently rewritten
\begin{equation} \label{eq:rec.op.int}
  \int_T\mathbb{A}_\eps\grad p_{\eps,T}^{k+1}(\uvT){\cdot}\grad w_\eps=\int_T\grad\vT{\cdot}\mathbb{A}_\eps\grad w_\eps-\int_{\partial T}(\vT-\vpT)\mathbb{A}_\eps\grad w_\eps{\cdot}\vec{n}_{\partial T}.
\end{equation}
Integrating by parts the left-hand side of~\eqref{eq:rec.op} and exploiting the definition~\eqref{eq:VkpoT} of the space $\VkpoT$, one can see that, for any $\uvT\in\UkT$,
\begin{equation} \label{eq:proj_reco}
\Pi_T^{k-1}\left(p_{\eps,T}^{k+1}(\uvT)\right)=\Pi_T^{k-1}(\vT)=\vT,\qquad\Pi_{\partial T}^k\left(p_{\eps,T}^{k+1}(\uvT)\right)=\Pi_{\partial T}^k(\vpT)=\vpT.
\end{equation}
Owing to~\eqref{eq:VkpoT} and~\eqref{eq:rec.op}, we infer that, for all $v\in H^1(T)$, 
  \begin{equation} \label{eq:ell.proj}
    \int_T\mathbb{A}_\eps\grad\left(v-p_{\eps,T}^{k+1}(\IkT v)\right){\cdot}\grad w_\eps=0\quad\forall w_\eps\in\VkpoT,
  \end{equation}
so that $p_{\eps,T}^{k+1}\circ \IkT:H^1(T) \to \VkpoT$ is the $\mathbb{A}_\eps$-weighted elliptic projection. As a consequence, we have, for all $v\in H^1(T)$, 
\begin{equation} \label{eq:inf}
    \norm{\L[T]^d}{\mathbb{A}_\eps^{\nicefrac{1}{2}}\grad\left(v-p_{\eps,T}^{k+1}(\IkT v)\right)}=\inf_{w_\eps\in\VkpoT}\norm{\L[T]^d}{\mathbb{A}_\eps^{\nicefrac{1}{2}}\grad\left(v-w_\eps\right)}.
\end{equation}
Since the operator $p_{\eps,T}^{k+1}\circ \IkT$ preserves the mean value, its restriction to $\VkpoT$ is the identity operator.

\begin{remark}[Comparison with the monoscale HHO method]
In the monoscale HHO method, the reconstruction operator is simpler to construct since it maps onto $\Pkpd{T}$ (which is a proper subspace of $\VkpoT$ whenever $\mathbb{A}_\eps$ is a constant matrix on $T$), whereas in the multiscale context, we explore the whole space $\VkpoT$ to build the reconstruction. One advantage of doing this is that we no longer need stabilization in the present case. Another advantage is that we recover the characterization of $p_{\eps,T}^{k+1}\circ \IkT$ as the $\mathbb{A}_\eps$-weighted elliptic projector onto $\VkpoT$, that is lost in the monoscale case as soon as $\mathbb{A}_\eps$ is not a constant matrix on $T$.
\end{remark}

The local bilinear form $a_{\eps,T}:\UkT\times\UkT\to\R$ is defined as
$$
a_{\eps,T}(\uuT,\uvT)\define\int_T\mathbb{A}_\eps\grad p_{\eps,T}^{k+1}(\uuT){\cdot}\grad p_{\eps,T}^{k+1}(\uvT).
$$
We introduce the following semi-norm on $\UkT$:
\begin{equation} \label{eq:normT}
  \norm{T}{\uvT}^2\define\norm{\L[T]^d}{\grad\vT}^2+\norm{\L[\partial T]}{H_{\partial T}^{-\nicefrac12}(\vT-\vpT)}^2.
\end{equation}

\begin{lemma}[Local stability] \label{le:stab}
The following holds:
\begin{equation}
\label{eq:stab_a}
a_{\eps,T}(\uvT,\uvT)\geq c\, \alpha\norm{T}{\uvT}^2\qquad \forall \uvT\in\UkT,
\end{equation}
with constant $c$ independent of $\eps$, $H_T$, $\alpha$ and $\beta$.
\end{lemma}

\begin{proof}
Let $\uvT\in\UkT$. To derive an estimate on $\norm{\L[T]^d}{\grad\vT}$, we define $v_\eps\in\VkpoT$ such that
  \begin{equation} \label{eq:pr.coer.1}
    \left\{
    \begin{alignedat}{3}
      -\div(\mathbb{A}_\eps\grad v_\eps)&=-\triangle\vT\in\mathbb{P}^{k-1}_d(T)&\qquad&\text{in $T$}, \\
      \mathbb{A}_\eps\grad v_\eps{\cdot}\vec{n}_{\partial T}&=\grad\vT{\cdot}\vec{n}_{\partial T}\in\Pkdmo{\FT}&\qquad&\text{on $\partial T$},
    \end{alignedat}
    \right.
  \end{equation}
  and satisfying, e.g., $\int_Tv_\eps=0$ (the way the constant is fixed is unimportant here). Note that data in~\eqref{eq:pr.coer.1} are compatible.
  Then, the following holds:
  \begin{equation*}
    \int_T\mathbb{A}_\eps\grad v_\eps{\cdot}\grad z=\int_T \grad\vT{\cdot}\grad z \qquad \forall z\in H^1(T).
  \end{equation*}
  Using this last relation where we take $z=p_{\eps,T}^{k+1}(\uvT)$, and using~\eqref{eq:rec.op.int} where we take $w_\eps=v_\eps\in \VkpoT$ defined in~\eqref{eq:pr.coer.1}, we infer that
\begin{align*}
-\int_T\vT\,\triangle\vT+\int_{\partial T}\vpT\,\grad\vT{\cdot}&\vec{n}_{\partial T}=
-\int_T\vT\,\div(\mathbb{A}_\eps\grad v_\eps)+\int_{\partial T}\vpT\,\mathbb{A}_\eps\grad v_\eps{\cdot}\vec{n}_{\partial T} \\
&=\int_T \mathbb{A}_\eps \grad v_\eps {\cdot} \grad\vT - \int_{\partial T}(\vT-\vpT)\mathbb{A}_\eps\grad v_\eps{\cdot}\vec{n}_{\partial T} \\
&= \int_T \mathbb{A}_\eps \grad v_\eps {\cdot} \grad p_{\eps,T}^{k+1}(\uvT)
= \int_T \grad\vT{\cdot}\grad p_{\eps,T}^{k+1}(\uvT).
\end{align*}
After an integration by parts, this yields
  $$\norm{\L[T]^d}{\grad\vT}^2=\int_T \grad p_{\eps,T}^{k+1}(\uvT){\cdot}\grad\vT+\int_{\partial T}(\vT-\vpT)\grad\vT{\cdot}\vec{n}_{\partial T}.$$
  By the Cauchy--Schwarz inequality and the discrete trace inequality~\eqref{eq:tr.d}, we then obtain
  \begin{equation} \label{eq:pr.coer.2}
    \norm{\L[T]^d}{\grad\vT}\leq c\left(\alpha^{-\nicefrac12}\norm{\L[T]^d}{\mathbb{A}_\eps^{\nicefrac12}\grad p_{\eps,T}^{k+1}(\uvT)}+\norm{\L[\partial T]}{H_{\partial T}^{-\nicefrac12}(\vT-\vpT)}\right).
  \end{equation}
To bound the second term in the right-hand side, we use~\eqref{eq:proj_reco} to infer that
\begin{align*}
[\vT-\vpT]_{\mid\partial T} &= 
[\Pi_T^{k-1}\left(p_{\eps,T}^{k+1}(\uvT)\right)]_{\mid\partial T}-\Pi_{\partial T}^k\left(p_{\eps,T}^{k+1}(\uvT)\right)
\\&=\Pi_{\partial T}^k\left(\Pi_T^{k-1}\left(p_{\eps,T}^{k+1}(\uvT)\right)-p_{\eps,T}^{k+1}(\uvT)\right).
\end{align*}
Using the $L^2$-stability of $\Pi_{\partial T}^k$, the continuous trace inequality~\eqref{eq:tr.c}, the local Poincar\'e inequality~\eqref{eq:poin} (since $p_{\eps,T}^{k+1}(\uvT)-\Pi^{k-1}_T\left(p_{\eps,T}^{k+1}(\uvT)\right)$ has zero mean-value on $T$), and the $H^1$-stability of $\Pi_T^{k-1}$, we infer that
  \begin{equation} \label{eq:pr.coer.4}
    \norm{\L[\partial T]}{H_{\partial T}^{-\nicefrac12}(\vT-\vpT)}\leq c\,\alpha^{-\nicefrac12}\norm{\L[T]^d}{\mathbb{A}_\eps^{\nicefrac12}\grad p_{\eps,T}^{k+1}(\uvT)}.
  \end{equation}
This concludes the proof. 
\end{proof}

We define the skeleton $\partial\TH$ of the mesh $\TH$ as $\partial\TH\define\bigcup_{F\in\FH}F$. We introduce the broken polynomial spaces
\begin{align}
  \Pkmd{\TH}&\define\left\{v\in\L\mid v_{\mid T}\in\Pkmd{T}\;\forall T\in\TH\right\},\\\Pkdmo{\FH}&\define\left\{v\in\L[\partial\TH]\mid v_{\mid F}\in\Pkdmo{F}\;\forall F\in\FH\right\}.\label{eq:def_Pkdmo_FH}
\end{align}
The global set of discrete unknowns is defined to be
\begin{equation} \label{eq:UkH}
  \UkH\define\Pkmd{\TH}\times\Pkdmo{\FH},
\end{equation}
so that any $\uvH\in\UkH$ can be decomposed as $\uvH\define\left(\vH,\vpH\right)$.
For any given $\uvH\in\UkH$, we denote $\uvT\define\left(\vT,\vpT\right)\in\UkT$ its restriction to the mesh cell $T\in\TH$.
Note that unknowns attached to mesh interfaces are single-valued, in the sense that, for any $F\in\FHi$ such that $F=\partial T_1\cap\partial T_2\cap Z$ for $T_1,T_2\in\TH$, $\vF\define{\rm v}_{\FH\mid F}\in\Pkdmo{F}$ is such that $\vF={\rm v}_{{\cal F}_{T_1}\mid F}={\rm v}_{{\cal F}_{T_2}\mid F}$.
To take into account homogeneous Dirichlet boundary conditions, we further introduce the subspace $\UkH[H,0]\define\left\{\uvH\in\UkH\mid\vF\equiv{\rm 0}\;\forall F\in\FHb\right\}$.
We define the global bilinear form $a_{\eps,H}:\UkH\times\UkH\to\R$ such that
$$
a_{\eps,H}(\uuH,\uvH)\define\sum_{T\in\TH}a_{\eps,T}(\uuT,\uvT) = \sum_{T\in\TH} \int_T\mathbb{A}_\eps\grad p_{\eps,T}^{k+1}(\uuT){\cdot}\grad p_{\eps,T}^{k+1}(\uvT).
$$
Then, the discrete problem reads: Find $\underline{{\rm u}}_{\eps,H}\in\UkH[H,0]$ such that
\begin{equation} \label{eq:dis.pro}
  a_{\eps,H}(\underline{{\rm u}}_{\eps,H},\uvH)=\int_\Omega f\vH\qquad\forall\uvH\in\UkH[H,0].
\end{equation}
Setting $\norm{H}{\uvH}^2\define\sum_{T\in\TH}\norm{T}{\uvT}^2$ on $\UkH[H]$, with $\norm{T}{{\cdot}}$ introduced in~\eqref{eq:normT}, this defines a norm on $\UkH[H,0]$ since elements in $\UkH[H,0]$ are such that $\vF\equiv{\rm 0}$ for all $F\in\FHb$.

\begin{lemma}[Well-posedness] \label{le:well}
The following holds, for all $\uvH\in\UkH$:
\begin{equation}
a_{\eps,H}(\uvH,\uvH) = \sum_{T\in\TH} \norm{L^2(T)^d}{\mathbb{A}_\eps^{\nicefrac12} \grad p_{\eps,T}^{k+1}(\uvT)}^2 =:\norm{\eps,H}{\uvH}^2\geq c\,\alpha\norm{H}{\uvH}^2,
\end{equation}
with constant $c$ independent of $\eps$, $H$, $\alpha$ and $\beta$.
As a consequence, the discrete problem~\eqref{eq:dis.pro} is well-posed.
\end{lemma}

\begin{proof}
This is a direct consequence of Lemma~\ref{le:stab}.
\end{proof}

\begin{remark}[Non-conforming Finite Element (ncFE) formulation] \label{re:nc-mixed}
Consider the discrete space 
  \begin{equation*}
    V^{k+1}_{\eps,H,0}\define\left\{v_{\eps,H}\in\L\mid v_{\eps,H\mid T}\in\VkpoT\;\forall\,T\in\TH,\,\PikF({\llbracket v_{\eps,H}\rrbracket}_F)=0\;\forall\,F\in\FH\right\},
  \end{equation*}
where ${\llbracket{\cdot}\rrbracket}_F$ denotes the jump operator for all interfaces $F\in\FHi$ (the sign is irrelevant) and the actual trace for all boundary faces $F\in\FHb$. Consider the following ncFE method:
Find $u_{\eps,H}\in V^{k+1}_{\eps,H,0}$ such that
\begin{equation}\label{eq:NcFE}
\tilde{a}_{\eps,H}(u_{\eps,H},v_{\eps,H})=\sum_{T\in\TH}\int_T f\,\Pi_T^{k-1}(v_{\eps,H})\qquad\forall v_{\eps,H}\in V^{k+1}_{\eps,H,0},
\end{equation}
where $\tilde{a}_{\eps,H}(u_{\eps,H},v_{\eps,H})\define\sum_{T\in\TH}\int_T\mathbb{A}_\eps\grad u_{\eps,H}{\cdot}\grad v_{\eps,H}$. 
Then, using that the restriction of $\IkT$ to $\VkpoT$ 
is an isomorphism from $\VkpoT$ to $\UkT$ and that  
the restriction of $p_{\eps,T}^{k+1}\circ \IkT$ to $\VkpoT$ is the 
identity operator, it can be shown that 
$\underline{{\rm u}}_{\eps,H}$ solves~\eqref{eq:dis.pro} if and only if
$\underline{{\rm u}}_{\eps,T}=\IkT(u_{\eps,H\mid T})$ for all $T\in\TH$
where $u_{\eps,H}$ solves~\eqref{eq:NcFE}. 
This proves that~\eqref{eq:dis.pro} is indeed a high-order extension of the method in~\cite{LBLLo:13}, up to a
different treatment of the right-hand side: $\Pi_T^{k-1}(v_{\eps,H})$ is used instead of $v_{\eps,H}$.
\end{remark}

Let $u_\eps$ be the oscillatory solution to~\eqref{eq:osc} and let $\underline{{\rm u}}_{\eps,H}$ be the discrete msHHO solution to~\eqref{eq:dis.pro}. Let us define the discrete error such that
\begin{equation} \label{eq:def.err}
\ueH \in \UkH[H,0], \qquad \ueT\define \IkT u_\eps - \underline{\rm u}_{\eps,T} \quad \forall T\in\TH.
\end{equation}
Note that $\ueH$ is well-defined as a member of $\UkH[H,0]$ since the oscillatory solution $u_\eps$ is in $H^1_0(\Omega)$ and functions in $H^1_0(\Omega)$ are single-valued at interfaces and vanish at the boundary.

\begin{lemma}[Discrete energy-error estimate] \label{lem:est}
Let the discrete error $\ueH$ be defined by~\eqref{eq:def.err}.
Assume that $u_0\in H^{k+2}(\Omega)$.
Then, the following holds:
\begin{multline} \label{eq:est}
\norm{\eps,H}{\ueH} \leq c\,\rho^{\nicefrac12}
\left(\beta \sum_{T\in\TH}H_T^{2(k+1)}\senorm{H^{k+2}(T)}{u_0}^2 \right.\\\left. + \sum_{T\in\TH}\norm{\L[T]^d}{\mathbb{A}_\eps^{\nicefrac12}\grad\left(u_\eps-\pi_{\eps,T}^{k+1}(u_0)\right)}^2\right)^{\nicefrac12},
\end{multline}
with constant $c$ independent of $\eps$, $H$, $u_0$, $\alpha$ and $\beta$.
\end{lemma}

\begin{proof}
Lemma~\ref{le:well} implies that
\begin{equation} \label{eq:pr.th.1}
\norm{\eps,H}{\ueH}=\sup_{\uvH\in\UkH[H,0]}\frac{a_{\eps,H}(\ueH,\uvH)}{\norm{\eps,H}{\uvH}}.
\end{equation}
Let $\uvH\in\UkH[H,0]$.
Performing an integration by parts, and using the facts that the flux $\mathbb{A}_0\grad u_0{\cdot}\vec{n}_F$ is continuous accross any interface $F\in\FHi$ since $u_0\in H^2(\Omega)$, and that $\uvH\in\UkH[H,0]$, we infer that
\begin{multline} \label{eq:pr.th.2}
    a_{\eps,H}(\underline{{\rm u}}_{\eps,H},\uvH)=\int_\Omega f\vH=\sum_{T\in\TH}\int_T\mathbb{A}_0\grad u_0{\cdot}\grad\vT\\-\sum_{T\in\TH}\int_{\partial T}(\vT-\vpT)\mathbb{A}_0\grad u_0{\cdot}\vec{n}_{\partial T}.
\end{multline}
Using~\eqref{eq:rec.op.int} with $w_\eps=p_{\eps,T}^{k+1}(\IkT u_\eps)$, we then infer that
\begin{multline*} 
a_{\eps,H}(\ueH,\uvH) =\sum_{T\in\TH}\int_T\left(\mathbb{A}_\eps\grad p_{\eps,T}^{k+1}(\IkT u_\eps)-\mathbb{A}_0\grad u_0\right){\cdot}\grad\vT\\-\sum_{T\in\TH}\int_{\partial T}\left(\mathbb{A}_\eps\grad p_{\eps,T}^{k+1}(\IkT u_\eps)-\mathbb{A}_0\grad u_0\right){\cdot}\vec{n}_{\partial T}(\vT-\vpT).
\end{multline*}
Adding/subtracting $\Pi^{k+1}_T (u_0)$ in the right-hand side yields
$a_{\eps,H}(\ueH,\uvH) = \term_1+\term_2$ with
\begin{align*}
 \term_1={}&\sum_{T\in\TH}\int_T\mathbb{A}_0\grad\left(\Pi^{k+1}_T (u_0)-u_0\right){\cdot}\grad\vT \\ &-\sum_{T\in\TH}\int_{\partial T}\mathbb{A}_0\grad\left(\Pi^{k+1}_T (u_0)-u_0\right){\cdot}\vec{n}_{\partial T}(\vT-\vpT),\\
\term_2={}&\sum_{T\in\TH}\int_T\left(\mathbb{A}_\eps\grad p_{\eps,T}^{k+1}(\IkT u_\eps)-\mathbb{A}_0\grad\Pi^{k+1}_T (u_0)\right){\cdot}\grad\vT\\
&-\sum_{T\in\TH}\int_{\partial T}\big(\mathbb{A}_\eps\grad p_{\eps,T}^{k+1}(\IkT u_\eps)-\mathbb{A}_0\grad\Pi^{k+1}_T (u_0)\big){\cdot}\vec{n}_{\partial T}(\vT-\vpT).
\end{align*}
The term $\term_1$ is estimated using Cauchy--Schwarz inequality and the approximation properties~\eqref{eq:app} of the projector $\Pi^{k+1}_T$ for $m=1$ and $s=k+2$, yielding
\begin{equation*} 
    \left|\term_1\right|\leq c\,\beta{\left(\sum_{T\in\TH}H_T^{2(k+1)}\senorm{H^{k+2}(T)}{u_0}^2\right)}^{\nicefrac12}\norm{H}{\uvH}.
\end{equation*}
Considering now $\term_2$, we use the definition~\eqref{eq:proj} of $\pi_{\eps,T}^{k+1}(u_0)$ and the relation~\eqref{eq:proj.wf} to infer that
\begin{align*}
    \term_2={}&\sum_{T\in\TH}\int_T\mathbb{A}_\eps\grad\left(p_{\eps,T}^{k+1}(\IkT u_\eps)-\pi_{\eps,T}^{k+1}(u_0)\right){\cdot}\grad\vT\\&-\sum_{T\in\TH}\int_{\partial T}\mathbb{A}_\eps\grad\left(p_{\eps,T}^{k+1}(\IkT u_\eps)-\pi_{\eps,T}^{k+1}(u_0)\right){\cdot}\vec{n}_{\partial T}(\vT-\vpT).
\end{align*}
The first term in the right-hand side can be bounded using the Cauchy--Schwarz inequality, whereas the second term is estimated by means of the inverse inequality from Lemma~\ref{lem:inv.VkpoT} since $\left(p_{\eps,T}^{k+1}(\IkT u_\eps)-\pi_{\eps,T}^{k+1}(u_0)\right)\in\VkpoT$. This yields
\begin{align*} 
    \left|\term_2\right|&\leq c\,\beta^{\nicefrac12}{\left(\sum_{T\in\TH}\norm{\L[T]^d}{\mathbb{A}_\eps^{\nicefrac12}\grad\left(p_{\eps,T}^{k+1}(\IkT u_\eps)-\pi_{\eps,T}^{k+1}(u_0)\right)}^2\right)}^{\nicefrac12}\norm{H}{\uvH} \\
&\leq c\,\beta^{\nicefrac12}{\left(\sum_{T\in\TH}\norm{\L[T]^d}{\mathbb{A}_\eps^{\nicefrac12}\grad\left(u_\eps-\pi_{\eps,T}^{k+1}(u_0)\right)}^2\right)}^{\nicefrac12}\norm{H}{\uvH},
\end{align*}
where the last bound follows from~\eqref{eq:inf} since $\pi_{\eps,T}^{k+1}(u_0)\in\VkpoT$.
Since $\norm{\eps,H}{\uvH}^2 \geq c\,\alpha\norm{H}{\uvH}^2$ owing to Lemma~\ref{le:well}, 
we obtain the expected bound.
\end{proof}

\begin{theorem}[Energy-error estimate] \label{th:est.err}
Assume that the correctors $\mu_{l}$ are in $W^{1,\infty}(\R^d)$ for any $1\leq l\leq d$, and that $u_0\in H^{k+2}(\Omega)$ (recall that $k\geq 1$).
Then, the following holds:
\begin{multline} \label{eq:est.2}
    \hspace{-0.35cm}{\left(\sum_{T\in\TH}\norm{L^2(T)^d}{\mathbb{A}_\eps^{\nicefrac12}\grad\left(u_\eps-p_{\eps,T}^{k+1}(\uueT)\right)}^2\right)}^{\nicefrac12}\leq c\,\beta^{\nicefrac12}\rho\bigg(
\sum_{T\in\TH}H_T^{2(k+1)}\senorm{H^{k+2}(T)}{u_0}^2 \\
+ \eps|\partial\Omega|\senorm{W^{1,\infty}(\Omega)}{u_0}^2
+ \sum_{T\in\TH} \left[\eps^2\senorm{H^2(T)}{u_0}^2+\,\eps|\partial T|\senorm{W^{1,\infty}(T)}{u_0}^2\right]\bigg)^{\nicefrac12},
\end{multline}
with $c$ independent of $\eps$, $H$, $u_0$, $\alpha$ and $\beta$. In particular, if the mesh $\TH$ is quasi-uniform, and tracking for simplicity only the dependency on $\eps$ and $H$ with $\eps\leq H\leq \ell_\Omega$ ($\ell_\Omega$ denotes the diameter of $\Omega$), we obtain an energy-error upper bound of the form $(\eps^{\nicefrac12} + H^{k+1} + (\eps/H)^{\nicefrac12})$.
\end{theorem}

\begin{proof}
Using the shorthand notation $e_{\eps,T}\define u_\eps|_T-p_{\eps,T}^{k+1}(\uueT)$ for all $T\in\TH$, the triangle inequality implies that
\begin{multline*}
{\left(\sum_{T\in\TH}\norm{L^2(T)^d}{\mathbb{A}_\eps^{\nicefrac12}\grad e_{\eps,T}}^2\right)}^{\nicefrac12} \leq {\left(\sum_{T\in\TH}\norm{L^2(T)^d}{\mathbb{A}_\eps^{\nicefrac12}\grad \left(u_\eps-p_{\eps,T}^{k+1}(\IkT u_\eps)\right)}^2\right)}^{\nicefrac12} \\+ \norm{\eps,H}{\ueH},
\end{multline*}
and owing to~\eqref{eq:inf}, we infer that
\begin{multline*}
{\left(\sum_{T\in\TH}\norm{L^2(T)^d}{\mathbb{A}_\eps^{\nicefrac12}\grad e_{\eps,T}}^2\right)}^{\nicefrac12} \leq {\left(\sum_{T\in\TH}\norm{L^2(T)^d}{\mathbb{A}_\eps^{\nicefrac12}\grad \left(u_\eps-\pi_{\eps,T}^{k+1}(u_0)\right)}^2\right)}^{\nicefrac12} \\+ \norm{\eps,H}{\ueH}.
\end{multline*}
Lemma~\ref{lem:est} then implies that
\begin{multline*}
{\left(\sum_{T\in\TH}\norm{L^2(T)^d}{\mathbb{A}_\eps^{\nicefrac12}\grad e_{\eps,T}}^2\right)}^{\nicefrac12} \leq \\ c\,\rho^{\nicefrac12}
{\left(\beta \sum_{T\in\TH}H_T^{2(k+1)}\senorm{H^{k+2}(T)}{u_0}^2 + \sum_{T\in\TH}\norm{\L[T]^d}{\mathbb{A}_\eps^{\nicefrac12}\grad\left(u_\eps-\pi_{\eps,T}^{k+1}(u_0)\right)}^2\right)}^{\nicefrac12}.
\end{multline*}
To conclude the proof of~\eqref{eq:est.2}, we add/subtract $\Leps(u_0)$ in the last term in the right-hand side, 
and invoke the triangle inequality together with Lemma~\ref{le:ts} to bound $(u_\eps-\Leps(u_0))$ globally on $\Omega$ and Lemma~\ref{lem:app.VkpoT} to bound $(\Leps(u_0)-\pi_{\eps,T}^{k+1}(u_0))$ locally on all $T\in\TH$. Finally, to derive the upper bound for quasi-uniform meshes, we observe that the last term in~\eqref{eq:est.2} can be estimated as $\sum_{T\in\TH} \eps|\partial T|\senorm{W^{1,\infty}(T)}{u_0}^2\le c\,\eps H^{-1}|u_0|_{W^{1,\infty}(\Omega)}^2 \sum_{T\in\TH} |\partial T| H_T \le c' \eps H^{-1}|u_0|_{W^{1,\infty}(\Omega)}^2$ with $c'$ proportional to $|\Omega|$.
\end{proof}

\begin{remark}[Dependency on $\rho$] \label{rem.rho}
The estimate~\eqref{eq:est.2} has a linear dependency with respect to the (global) heterogeneity/anisotropy ratio $\rho$ (a close inspection of the proof shows that the term $\eps^{\nicefrac12}|\partial\Omega|^{\nicefrac12}\senorm{W^{1,\infty}(\Omega)}{u_0}$ only scales with $\rho^{\nicefrac12}$). This linear scaling is also obtained with the monoscale HHO method when the diffusivity is non-constant in each mesh cell; cf.~\cite[Theorem 3.1]{DPELe:16}.
\end{remark}

\begin{remark}[Discretization of the right-hand side] \label{rem:disc_rhs}
Note that we could also integrate the right-hand side in~\eqref{eq:dis.pro} using $p_{\eps,T}^{k+1}(\uvT)$ instead of $\vT$ on each $T\in\TH$, up to the addition in the right-hand sides of the bounds~\eqref{eq:est} and~\eqref{eq:est.2} of the optimally convergent term $c\,\alpha^{-\nicefrac12}{\left(\sum_{T\in\TH}H_T^{2(k+1)}\senorm{H^k(T)}{f}^2\right)}^{\nicefrac12}$.
Indeed, owing to~\eqref{eq:proj_reco}, we have
$$\sum_{T\in\TH}\int_T f\,(\vT-p_{\eps,T}^{k+1}(\uvT))=\sum_{T\in\TH}\int_T (f-\Pi^{k-1}_T(f))\,(\vT-p_{\eps,T}^{k+1}(\uvT)),$$
which can be estimated by applying Cauchy--Schwarz inequality on each $T$, and \begin{inparaenum} \item[(i)] the approximation properties~\eqref{eq:app} of $\Pi^{k-1}_T$ with $m=0$ and $s=k$ for the first factor, \item[(ii)] the Poincar\'e inequality~\eqref{eq:poin} (recall that $(\vT-p_{\eps,T}^{k+1}(\uvT))$ has zero-mean on $T$) and the triangle inequality combined with Lemma~\ref{le:stab} for the second factor. \end{inparaenum}  
This alternative approach, that is pursued in~\cite{LBLLo:13,LBLLo:14}, necessitates an integration against oscillatory test functions. It is hence computationally more expensive (recall that $f$ is assumed to be non-oscillatory), and may become limiting in a multi-query context.
\end{remark}

\subsection{The equal-order case}
\label{sec:equal_order}

Let $k\geq0$. For all $T\in\TH$, we consider now the following local set of discrete unknowns:

\begin{equation} \label{eq:UkT.equal}
  \UkT\define\Pkd{T}\times\Pkdmo{\FT}.
\end{equation}
Any element $\uvT\in\UkT$ is again decomposed as $\uvT\define(\vT,\vpT)$, and for any $F\in\FT$, we denote $\vF\define{\rm v}_{\FT\mid F}\in\Pkdmo{F}$.
We redefine the local reduction operator $\IkT:H^1(T)\to\UkT$ so that, for any $v\in H^1(T)$, $\IkT v\define(\Pi^k_T(v),\Pi^k_{\partial T}(v))$.
Reasoning as in \cite[Section 2.4]{CoDPE:16}, 
it can be proved that, for all $T\in\TH$, the restriction of $\IkT$ to $\tilde{V}^{k+1}_{\eps,T}$ is an isomorphism from $\tilde{V}^{k+1}_{\eps,T}$ to $\UkT$, where
\begin{equation}
    \tilde{V}^{k+1}_{\eps,T}\define\left\{v_\eps\in H^1(T)\mid\div(\mathbb{A}_\eps\grad v_\eps)\in\mathbb{P}^k_d(T),\,\mathbb{A}_\eps\grad v_\eps{\cdot}\vec{n}_{\partial T}\in\Pkdmo{\FT}\right\}.
\end{equation}
Thus, the triple $(T,\tilde{V}^{k+1}_{\eps,T},\IkT)$ defines a finite element in the sense of Ciarlet.

The local multiscale reconstruction operator $p_{\eps,T}^{k+1}:\UkT\to\VkpoT$ is still defined as in~\eqref{eq:rec.op}, so that the key relations~\eqref{eq:ell.proj} and~\eqref{eq:inf} still hold. In particular, $p_{\eps,T}^{k+1}\circ \IkT:H^1(T) \to \VkpoT$ is the $\mathbb{A}_\eps$-weighted elliptic projection. However, the restriction of $p_{\eps,T}^{k+1}\circ \IkT$ to the larger space
$\tilde{V}^{k+1}_{\eps,T}$ is {\em not} the identity operator since $p_{\eps,T}^{k+1}$ maps onto the smaller space $\VkpoT$.
Concerning~\eqref{eq:proj_reco}, we still have $\Pi_{\partial T}^k\left(p_{\eps,T}^{k+1}(\uvT)\right)=\vpT$, but now $\Pi_T^{k-1}\left(p_{\eps,T}^{k+1}(\uvT)\right)=\Pi_T^{k-1}(\vT)$ is in general different from $\vT$.
 
This leads us to introduce the 
symmetric, positive semi-definite stabilization
\begin{equation} \label{eq:stab.eps}
j_{\eps,T}(\uuT,\uvT)\define\alpha\int_{\partial T}H^{-1}_{\partial T}\left(\uT-\Pi^k_T\left(p_{\eps,T}^{k+1}(\uuT)\right)\right)\left(\vT-\Pi^k_T\left(p_{\eps,T}^{k+1}(\uvT)\right)\right).
\end{equation}
The local bilinear form $a_{\eps,T}:\UkT\times\UkT\to\R$ is then defined as
$$
a_{\eps,T}(\uuT,\uvT)\define\int_T\mathbb{A}_\eps\grad p_{\eps,T}^{k+1}(\uuT){\cdot}\grad p_{\eps,T}^{k+1}(\uvT)+j_{\eps,T}(\uuT,\uvT).
$$

\begin{remark}[Variant]
Alternatively, one can discard the stabilization at the price of computing additional cell-based oscillatory basis functions, using the basis functions ${(\Phi_T^{k,i})}_{1\leq i\leq {\rm N}^k_d}$ instead of ${(\Phi_T^{k-1,i})}_{1\leq i\leq {\rm N}^{k-1}_d}$ as proposed in Section~\ref{ssse:bas}. This is the approach pursued in~\cite{LBLLo:14} for $k=0$ where one cell-based oscillatory basis function is added (in the slightly different context of perforated domains). 
\end{remark}

Recall the local stability semi-norm $\norm{T}{\cdot}$ defined by~\eqref{eq:normT}.
  
\begin{lemma}[Local stability and approximation] \label{lemma.sta}
The following holds:
\begin{equation} \label{eq:stab_j}
a_{\eps,T}(\uvT,\uvT)\geq c\,\alpha\norm{T}{\uvT}^2\qquad \forall \uvT\in\UkT.
\end{equation}
Moreover, for all $v\in H^1(T)$,
  \begin{equation} \label{eq:jeps}
    j_{\eps,T}(\IkT v,\IkT v)^{\nicefrac12}\leq c\,\norm{\L[T]^d}{\mathbb{A}_\eps^{\nicefrac{1}{2}}\grad\left(v-p_{\eps,T}^{k+1}(\IkT v)\right)},
  \end{equation}
  with (distinct) constants $c$ independent of $\eps$, $H_T$, $\alpha$ and $\beta$.
\end{lemma}
\begin{proof}
To prove stability, we adapt the proof of Lemma~\ref{le:stab}.
Let $\uvT\in\UkT$. The bound~\eqref{eq:pr.coer.2} on $\norm{\L[T]^d}{\grad\vT}$ still
holds, so that we only need to bound $\norm{\L[\partial T]}{H_{\partial T}^{-\nicefrac12}(\vT-\vpT)}$. Since $\Pi_{\partial T}^k\left(p_{\eps,T}^{k+1}(\uvT)\right)=\vpT$, we infer that
$(\vT-\vpT) = \Pi_{\partial T}^k\left(\vT-p_{\eps,T}^{k+1}(\uvT)\right)$, so that invoking the $L^2$-stability of $\Pi_{\partial T}^k$ and the triangle inequality while adding/subtracting $\Pi^k_T\left(p_{\eps,T}^{k+1}(\uvT)\right)$, we obtain
\begin{align*}
\norm{\L[\partial T]}{H_{\partial T}^{-\nicefrac12}(\vT-\vpT)} \leq {}&\norm{\L[\partial T]}{H_{\partial T}^{-\nicefrac12}\left(\vT-\Pi^k_T\left(p_{\eps,T}^{k+1}(\uvT)\right)\right)} \\&
+ \norm{\L[\partial T]}{H_{\partial T}^{-\nicefrac12}\left(p_{\eps,T}^{k+1}(\uvT)-\Pi^k_T\left(p_{\eps,T}^{k+1}(\uvT)\right)\right)}.
\end{align*}
The first term in the right-hand side is bounded by $\alpha^{-\nicefrac12}j_{\eps,T}(\uvT,\uvT)^{\nicefrac12}$, and the second one has been bounded (with the use of $\Pi^{k-1}_T$ instead of $\Pi^k_T$) in the proof of Lemma~\ref{le:stab} (see~\eqref{eq:pr.coer.4}) by  
$c\,\alpha^{-\nicefrac12}\norm{\L[T]^d}{\mathbb{A}_\eps^{\nicefrac12}\grad p_{\eps,T}^{k+1}(\uvT)}$.
To prove~\eqref{eq:jeps}, we start from
$$j_{\eps,T}(\IkT v,\IkT v)=\alpha\norm{\L[\partial T]}{H_{\partial T}^{-\nicefrac12}\Pi^k_T\left(v-p_{\eps,T}^{k+1}(\IkT v)\right)}^2.$$
The result then follows from the application of the discrete trace inequality~\eqref{eq:tr.d}, the $L^2$-stability property of $\Pi^k_T$, and the local Poincar\'e inequality~\eqref{eq:poin} (since $\int_Tp_{\eps,T}^{k+1}(\IkT v)=\int_T v$).
\end{proof}
We define the broken polynomial space
\begin{equation*}
  \Pkd{\TH}\define\left\{v\in\L\mid v_{\mid T}\in\Pkd{T}\;\forall T\in\TH\right\}, 
\end{equation*}
and the global set of discrete unknowns is defined to be
\begin{equation} \label{eq:UkH.equal}
  \UkH\define\Pkd{\TH}\times\Pkdmo{\FH},
\end{equation}
where $\Pkdmo{\FH}$ is still defined by~\eqref{eq:def_Pkdmo_FH}.
To take into account homogeneous Dirichlet boundary conditions, we consider again the subspace $\UkH[H,0]\define\big\{\uvH\in\UkH\mid\vF\equiv{\rm 0}\;\forall F\in\FHb\big\}$.
We define the global bilinear form $a_{\eps,H}:\UkH\times\UkH\to\R$ such that
\begin{multline*}
  a_{\eps,H}(\uuH,\uvH)\define\sum_{T\in\TH}a_{\eps,T}(\uuT,\uvT)= \sum_{T\in\TH} \bigg( \int_T\mathbb{A}_\eps\grad p_{\eps,T}^{k+1}(\uuT){\cdot}\grad p_{\eps,T}^{k+1}(\uvT) \\ + j_{\eps,T}(\uuT,\uvT)\bigg).
\end{multline*}
Then, the discrete problem reads: Find $\underline{{\rm u}}_{\eps,H}\in\UkH[H,0]$ such that
\begin{equation} \label{eq:dis.pro.equal}
  a_{\eps,H}(\underline{{\rm u}}_{\eps,H},\uvH)=\int_\Omega f\vH\qquad\forall\uvH\in\UkH[H,0].
\end{equation}
Recalling the norm $\norm{H}{\uvH}^2\define\sum_{T\in\TH}\norm{T}{\uvT}^2$ on $\UkH[H,0]$, we readily infer from Lemma~\ref{lemma.sta} the following well-posedness result.

\begin{lemma}[Well-posedness] \label{le:well.equal}
  The following holds, for all $\uvH\in\UkH$:
\begin{align}
a_{\eps,H}(\uvH,\uvH) &=\sum_{T\in\TH} \left( \norm{L^2(T)^d}{\mathbb{A}_\eps^{\nicefrac12} \grad p_{\eps,T}^{k+1}(\uvT)}^2 + j_{\eps,T}(\uvT,\uvT)\right) \nonumber\\
&=:\norm{\eps,H}{\uvH}^2\geq c\,\alpha\norm{H}{\uvH}^2,
\end{align}
with constant $c$ independent of $\eps$, $H$, $\alpha$ and $\beta$.
As a consequence, the discrete problem~\eqref{eq:dis.pro.equal} is well-posed.
\end{lemma}

\begin{remark}[ncFE interpretation] \label{re:nc-equal}
  As in Remark~\ref{re:nc-mixed}, it is possible to give a ncFE interpretation of the scheme~\eqref{eq:dis.pro.equal}. Let
  \begin{equation*}
    \tilde{V}^{k+1}_{\eps,H,0}\define\left\{v_{\eps,H}\in\L\mid v_{\eps,H\mid T}\in \tilde{V}^{k+1}_{\eps,T}\;\forall\,T\in\TH,\,\PikF({\llbracket v_{\eps,H}\rrbracket}_F)=0\;\forall\,F\in\FH\right\},
  \end{equation*}
and consider the following ncFE method: Find $u_{\eps,H}\in\tilde{V}^{k+1}_{\eps,H,0}$ such that
  \begin{equation} \label{eq:NcFE_equal}
    \tilde{a}_{\eps,H}(u_{\eps,H},v_{\eps,H})=\sum_{T\in\TH}\int_T f\,\PikT(v_{\eps,H})\qquad\forall v_{\eps,H}\in\tilde{V}^{k+1}_{\eps,H,0},
  \end{equation}
where $\tilde{a}_{\eps,H}(u_{\eps,H},v_{\eps,H})\define\sum_{T\in\TH}a_{\eps,T}\left(\IkT(u_{\eps,H\mid T}),\IkT(v_{\eps,H\mid T})\right)$.
Then, it can be shown that $\underline{{\rm u}}_{\eps,H}$ solves~\eqref{eq:dis.pro.equal} if and only if $\underline{{\rm u}}_{\eps,T}=\IkT(u_{\eps,H\mid T})$ for all $T\in\TH$ where $u_{\eps,H}$ solves~\eqref{eq:NcFE_equal}.
The main difference with respect to the mixed-order case is that it is no longer
possible to simplify the expression of the bilinear form $\tilde{a}_{\eps,H}$
since the restriction of $p_{\eps,T}^{k+1}\circ \IkT$ to 
$\tilde{V}^{k+1}_{\eps,T}$ is not the identity operator.
As in the monoscale HHO method, the operator $p_{\eps,T}^{k+1}$,
which maps onto the smaller space $\VkpoT$, allows one
to restrict the number of computed basis functions while maintaining
optimal (and here also $\eps$-robust) approximation properties.
The functions (from the discrete space $\tilde{V}_{\eps,T}^{k+1}$) that are eliminated (not computed) are handled by the stabilization term.
\end{remark}

\begin{lemma}[Discrete energy-error estimate] \label{lem:est.equal}
Let the discrete error $\ueH$ be defined by~\eqref{eq:def.err}.
Assume that $u_0\in H^{k+2}(\Omega)$.
Then, the following holds:
\begin{multline} \label{eq:est.equal}
\norm{\eps,H}{\ueH} \leq c\,\rho^{\nicefrac12}
\left(\beta \sum_{T\in\TH}H_T^{2(k+1)}\senorm{H^{k+2}(T)}{u_0}^2 \right.\\+\left. \sum_{T\in\TH}\norm{\L[T]^d}{\mathbb{A}_\eps^{\nicefrac12}\grad\left(u_\eps-\pi_{\eps,T}^{k+1}(u_0)\right)}^2\right)^{\nicefrac12},
\end{multline}
with constant $c$ independent of $\eps$, $H$, $u_0$, $\alpha$ and $\beta$.
\end{lemma}

\begin{proof}
The only difference with the proof of Lemma~\ref{lem:est} is that we now have
$a_{\eps,H}(\ueH,\uvH) = \term_1+\term_2+\term_3$, where $\term_1,\term_2$ are defined and bounded
in that proof and where 
\[
\term_3 \define \sum_{T\in\TH}j_{\eps,T}(\IkT u_\eps,\uvT).
\]
Since $j_{\eps,T}$ is symmetric, positive semi-definite, we infer that
\begin{align*} 
    \left|\term_3\right|&\leq{\left(\sum_{T\in\TH}j_{\eps,T}\left(\IkT u_\eps,\IkT u_\eps\right)\right)}^{\nicefrac12}{\left(\sum_{T\in\TH}j_{\eps,T}(\uvT,\uvT)\right)}^{\nicefrac12}\\
&\leq c{\left(\sum_{T\in\TH}\norm{\L[T]^d}{\mathbb{A}_\eps^{\nicefrac12}\grad\left(u_\eps-p_{\eps,T}^{k+1}(\IkT u_\eps)\right)}^2\right)}^{\nicefrac12}\norm{\eps,H}{\uvH},
\end{align*}
where we have used~\eqref{eq:jeps}. We can now conclude as before.
\end{proof}

\begin{theorem}[Energy-error estimate] \label{th:est.err.equal}
Assume that the correctors $\mu_{l}$ are in $W^{1,\infty}(\R^d)$ for any $1\leq l\leq d$, and that $u_0\in H^{k+2}(\Omega)\cap W^{1,\infty}(\Omega)$.
Then, the following holds:
\begin{multline} \label{eq:est.2.equal}
    \hspace{-0.35cm}{\left(\sum_{T\in\TH}\norm{L^2(T)^d}{\mathbb{A}_\eps^{\nicefrac12}\grad\left(u_\eps-p_{\eps,T}^{k+1}(\uueT)\right)}^2\right)}^{\nicefrac12}\leq c\,\beta^{\nicefrac12}\rho\bigg(
\sum_{T\in\TH}H_T^{2(k+1)}\senorm{H^{k+2}(T)}{u_0}^2 \\
+ \eps|\partial\Omega|\senorm{W^{1,\infty}(\Omega)}{u_0}^2
+ \sum_{T\in\TH} \left[\eps^2\senorm{H^2(T)}{u_0}^2+\,\eps|\partial T|\senorm{W^{1,\infty}(T)}{u_0}^2\right]\bigg)^{\nicefrac12},
\end{multline}
with $c$ independent of $\eps$, $H$, $u_0$, $\alpha$ and $\beta$. In particular, if the mesh $\TH$ is quasi-uniform, and tracking for simplicity only the dependency on $\eps$ and $H$ with $\eps\leq H\leq \ell_\Omega$, we obtain an energy-error upper bound of the form $(\eps^{\nicefrac12} + H^{k+1} + (\eps/H)^{\nicefrac12})$.
\end{theorem}

\begin{proof}
Identical to that of Theorem~\ref{th:est.err}.
\end{proof}

\begin{remark}[Dependency on $\rho$]
As in the mixed-order case (cf.~Remark~\ref{rem.rho}), the estimate~\eqref{eq:est.2.equal} has a linear dependency with respect to the (global) heterogeneity/anisotropy ratio $\rho$.
\end{remark}

\begin{remark}[Discretization of the right-hand side] \label{rem:disc_rhs_2}
The same observation as in Remark~\ref{rem:disc_rhs} concerning the discretization of the right-hand side in~\eqref{eq:dis.pro.equal} is still valid for the equal-order case.
\end{remark}

\section{Numerical results} \label{se:num.val}

In this section, we discuss the organization of the computations and we present some numerical results illustrating the above analysis for both the mixed-order and equal-order msHHO methods. Our numerical results have been obtained using the \texttt{Disk++} library, which is available as open-source under MPL license at the address \url{https://github.com/datafl4sh/diskpp}. The numerical core of the library is described in~\cite{CiDPE:17}. For the numerical tests presented below, we have used the direct solver PARDISO of the Intel MKL library. The simulations were run on an Intel i7-3615QM (2.3GHz) with 16Gb of RAM.

\subsection{Offline/online solution strategy} \label{sse:off.on}

Let us consider the equal-order version ($k\geq 0$) of the msHHO method introduced in Section~\ref{sec:equal_order}. Similar considerations carry over to the mixed-order case ($k\geq 1$) of Section~\ref{sec:mixed_order}.
To solve~\eqref{eq:dis.pro.equal}, we adopt an offline/online strategy.
\begin{itemize}
  \item[$\bullet$] In the offline step, all the computations are local, and independent of the right-hand side $f$. We first compute the cell-based and face-based basis functions, i.e., for all $T\in\TH$, we compute the ${\rm N}^{k-1}_d$ functions $\varphi^{k+1,i}_{\eps,T}$ solution to~\eqref{eq:bas.cell.pr}, and the ${\rm card}(\FT)\times{\rm N}^k_{d-1}$ functions $\varphi^{k+1,j}_{\eps,T,F}$ solution to~\eqref{eq:bas.face.pr} (cf.~Remark~\ref{re:prac}). This first substep is fully parallelizable. In a second time, we compute the multiscale reconstruction operators $p_{\eps,T}^{k+1}$, by solving~\eqref{eq:rec.op} for all $T\in\TH$. Each computation requires to invert a symmetric positive-definite matrix of size $\left({\rm N}^{k-1}_d+{\rm card}(\FT)\times{\rm N}^k_{d{-}1}\right)$, which can be performed effectively via Cholesky factorization. This second substep is as well fully parallelizable.
    Finally, we perform static condensation locally in each cell of $\TH$, to eliminate the cell unknowns. Details can be found in~\cite[Section~3.3.1]{DPELe:16}. Basically, in each cell, this substep consists in inverting a symmetric positive-definite matrix of size ${\rm N}^{k}_d$. This last substep is also fully parallelizable.
  \item[$\bullet$] In the online step, we compute the $L^2$-orthogonal projection of the right-hand side $f$ onto $\mathbb{P}^{k}_d(\TH)$, and we then solve a symmetric positive-definite global problem, posed in terms of the face unknowns only. The size of this problem is ${\rm card}(\FHi)\times{\rm N}^k_{d-1}$. If one wants to compute an approximation of the solution to~\eqref{eq:osc} for another $f$ (or for other boundary conditions), only the online step must be rerun.
\end{itemize}

\subsection{Periodic test-case} \label{sse:per}

We consider the periodic test-case studied in~\cite{LBLLo:13} (and also in~\cite{PaVaV:17}).
We let $d=2$, and $\Omega$ be the unit square. We consider Problem~\eqref{eq:osc}, with right-hand side $f(x,y)=\sin(x)\sin(y)$, and oscillatory coefficient
\begin{equation} \label{eq:tc.per}
  \mathbb{A}_\eps(x,y)=a(x/\eps,y/\eps)\mathbb{I}_2,\qquad a(x_1,x_2)=1+100\cos^2(\pi x_1)\sin^2(\pi x_2).
\end{equation}
For the coefficient~\eqref{eq:tc.per}, the homogenized tensor is given by $\mathbb{A}_0\approx 6.72071\,\mathbb{I}_2$.
We fix $\eps=\pi/150\approx 0.021$. 

\begin{figure}[htb]
  \centering
  \ifCMAM
  \epsfig{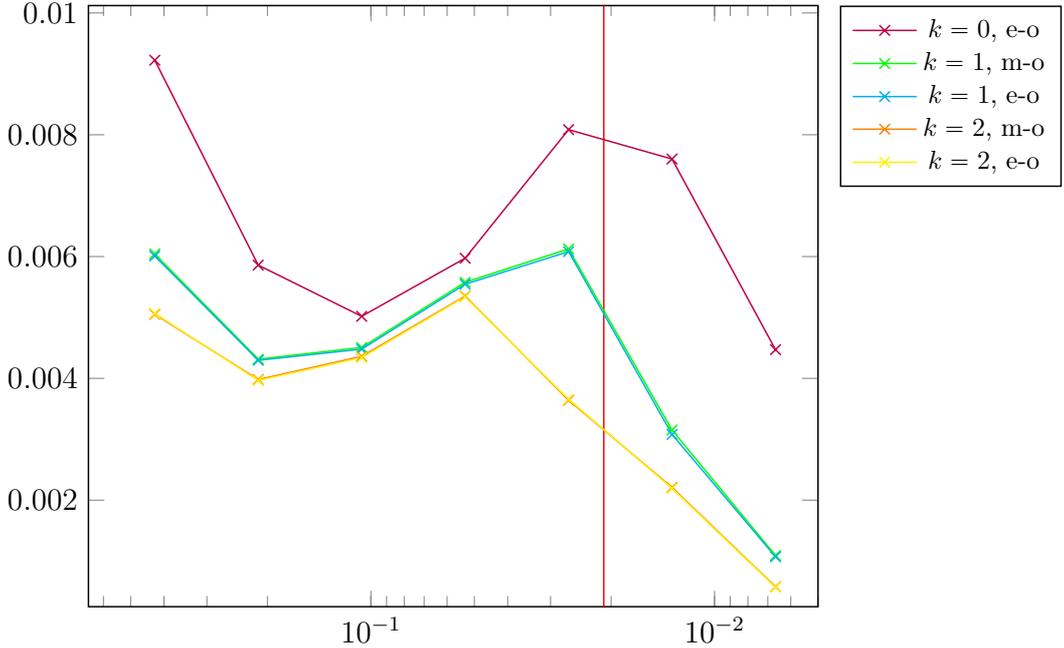}
  \else
  \begin{tikzpicture}[scale=1.4,every node/.style={scale=0.7}]
    \begin{semilogxaxis}[
        legend style = {
          legend pos = outer north east
        },
        yticklabel style={
            /pgf/number format/fixed,
            /pgf/number format/precision=5
        },
        scaled y ticks=false,
        xmin=0.005,
        ymin=0.00025,
        x dir = reverse
      ]
      \draw[color=red] ({axis cs:0.021,0}|-{rel axis cs:0,0}) -- ({axis cs:0.021,0}|-{rel axis cs:0,1});
      \addplot[mark=x,color=purple] table[x=H,y=K0k1eo]{mshhodata_002_ref_L9_k2.dat};
      \addplot[mark=x,color=green] table[x=H,y=K1k1mo]{mshhodata_002_ref_L9_k2.dat};
      \addplot[mark=x,color=cyan] table[x=H,y=K1k1eo]{mshhodata_002_ref_L9_k2.dat};
      \addplot[mark=x,color=orange] table[x=H,y=K2k1mo]{mshhodata_002_ref_L9_k2.dat};
      \addplot[mark=x,color=yellow] table[x=H,y=K2k1eo]{mshhodata_002_ref_L9_k2.dat};
      \legend{{\footnotesize $k=0$, e-o},{\footnotesize $k=1$, m-o},{\footnotesize $k=1$, e-o},{\footnotesize $k=2$, m-o},{\footnotesize $k=2$, e-o}};
    \end{semilogxaxis}
  \end{tikzpicture}
  \fi
  \caption{Periodic test-case: convergence results in energy-norm for mesh levels $l\in\{0,\ldots,6\}$; mixed-order msHHO method with polynomial degrees $k\in\{1,2\}$ and equal-order msHHO method with polynomial degrees $k\in\{0,1,2\}$. The red vertical line indicates the value of $\eps$.}
  \label{fig:conv_1_1_ref9_2}
\end{figure}

We consider a sequence of hierarchical triangular meshes of size $H_l=0.43\times 2^{-l}$ with $l\in\{0,\ldots,9\}$, so that $H_5<\eps<H_4$. A reference solution is computed by solving~\eqref{eq:osc} with the (equal-order) monoscale HHO method on the mesh of level $l_{\textup{ref}}=9$ with polynomial degree $k_{\textup{ref}}=2$. In Figure~\ref{fig:conv_1_1_ref9_2}, we present the (absolute) energy-norm errors obtained with the msHHO method on the meshes ${\cal T}_{H_l}$ with $l\in\{0,\ldots,6\}$. We consider both the mixed-order msHHO method with polynomial degrees $k\in\{1,2\}$ and the equal-order msHHO method with polynomial degrees $k\in\{0,1,2\}$. In all cases, the cell- and face-based oscillatory basis functions are precomputed using the (equal-order) monoscale HHO method on the mesh of level $l_{\textup{osc}}=8$ with polynomial degree $k_{\textup{osc}}=1$. We have verified that the oscillatory basis functions are sufficiently well resolved by comparing our results to those obtained with $k_{\textup{osc}}=2$ and obtaining only very marginal differences. The first observation we draw from Figure~\ref{fig:conv_1_1_ref9_2} is that the mixed-order and equal-order msHHO methods employing the same polynomial degree for the face unknowns deliver very similar results; indeed, the error curves are barely distinguishable both for $k=1$ and $k=2$. Moreover, we can observe all the main features expected from the error analysis: a pre-asymptotic regime where the term $H^{k+1}$ essentially dominates (meshes of levels $l\in\{0,1\}$), the resonance regime (meshes of levels $l\in\{2,3,4\}$ essentially), and the asymptotic regime where the mesh actually resolves the fine scale of the model coefficients (meshes of levels $l\in\{5,6\}$). We can also see the advantages of using a higher polynomial order for the face unknowns: the error is overall smaller, the minimal error in the resonance regime is reached at a larger value of $H$ and takes a smaller value (incidentally, the maximal error in the resonance regime takes a smaller value as well), and the asymptotic regime starts for larger values of $H$.

\subsection{Locally periodic test-case}

\begin{figure}[htb]
  \centering
  \ifCMAM
  \epsfig{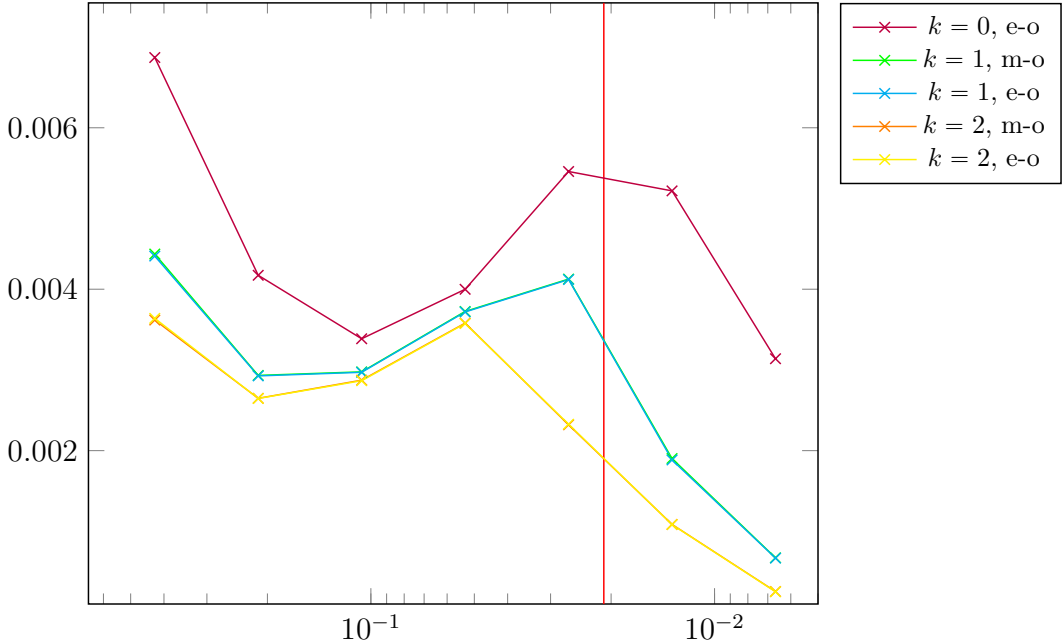}
  \else
  \begin{tikzpicture}[scale=1.4,every node/.style={scale=0.7}]
    \begin{semilogxaxis}[
        legend style = {
          legend pos = outer north east
        },
        yticklabel style={
            /pgf/number format/fixed,
            /pgf/number format/precision=5
        },
        scaled y ticks=false,
        xmin=0.005,
        ymin=0.0001,
        x dir = reverse
      ]
      \draw[color=red] ({axis cs:0.021,0}|-{rel axis cs:0,0}) -- ({axis cs:0.021,0}|-{rel axis cs:0,1});
      \addplot[mark=x,color=purple] table[x=H,y=K0k1eo]{mshhodata_002_ref_L9_k2_lp.dat};
      \addplot[mark=x,color=green] table[x=H,y=K1k1mo]{mshhodata_002_ref_L9_k2_lp.dat};
      \addplot[mark=x,color=cyan] table[x=H,y=K1k1eo]{mshhodata_002_ref_L9_k2_lp.dat};
      \addplot[mark=x,color=orange] table[x=H,y=K2k1mo]{mshhodata_002_ref_L9_k2_lp.dat};
      \addplot[mark=x,color=yellow] table[x=H,y=K2k1eo]{mshhodata_002_ref_L9_k2_lp.dat};
      \legend{{\footnotesize $k=0$, e-o},{\footnotesize $k=1$, m-o},{\footnotesize $k=1$, e-o},{\footnotesize $k=2$, m-o},{\footnotesize $k=2$, e-o}};
    \end{semilogxaxis}
  \end{tikzpicture}
  \fi
  \caption{Locally periodic test-case: convergence results in energy-norm for mesh levels $l\in\{0,\ldots,6\}$; mixed-order msHHO method with polynomial degrees $k\in\{1,2\}$ and equal-order msHHO method with polynomial degrees $k\in\{0,1,2\}$. The red vertical line indicates the value of $\eps$.}
  \label{fig:conv_1_1_ref9_2_lp}
\end{figure}

Keeping the same two-dimensional domain $\Omega$ as in the periodic test-case of Section~\ref{sse:per}, we consider now a locally periodic test-case where
we solve Problem~\eqref{eq:osc} with unchanged right-hand side $f(x,y)=\sin(x)\sin(y)$, but with oscillatory coefficient
\begin{equation} \label{eq:tc.alm.per}
  \mathbb{A}_\eps(x,y)=\left(a(x/\eps,y/\eps)+{\rm e}^{(x^2+y^2)/2}\right)\mathbb{I}_2,\quad\text{with $a$ given in~\eqref{eq:tc.per},}
\end{equation}
and with unchanged value of $\eps$. 
We perform the same numerical experiments as in Section~\ref{sse:per} using the same mesh level and polynomial order parameters for computing the reference solution and the oscillatory basis functions (we verified similarly the adequate resolution of the oscillatory basis functions). Results are reported in Figure~\ref{fig:conv_1_1_ref9_2_lp}. We can draw the same conclusions as in the periodic test-case: similarity of the results delivered by the mixed-order and the equal-order msHHO methods for both $k=1$ and $k=2$, presence of the pre-asymptotic, resonance, and asymptotic regimes, and advantages of using a higher polynomial order for the face unknowns.

To briefly assess computational costs, we compute, for those mesh levels in the pre-asymptotic or resonance regimes for which the error is minimal, the computational times to perform the offline and online steps. We report the results in Table~\ref{ta:time}. We also report the number of degrees of freedom in the global system solved in the online step. We make the experiment for the equal-order msHHO method of orders $k=0$ and $k=2$, for respective mesh levels $l=2$ and $l=1$. We do not make use of parallelism in our implementation to compute the results. Table~\ref{ta:time} shows the interest of higher-order approximations, since a better accuracy is reached at a smaller online computational cost.

\begin{table}
  \caption{\label{ta:time} Offline and online computational times}
  \centering
  \begin{tabular}{|c|c|c|c|c|}
    \hline
     & Energy-error & Offline time (s) & Online time (s) & \#DoFs \\
    \hline\hline
    $k=0\;(l=2)$ & $0.00338612$ & $254$ & $0.026$ & $408$ \\
    $k=2\;(l=1)$ & $0.00264648$ & $520$ & $0.018$ & $288$ \\
    \hline
  \end{tabular}
\end{table}

\paragraph*{Acknowledgments}

The authors are thankful to Alexei Lozinski (LMB, Universit\'e de Franche-Comt\'e) for fruitful discussions on the topic.

\appendix

\section{Estimates on the first-order two-scale expansion} \label{ap:est}

In this appendix, we derive various useful estimates on the first-order two-scale expansion 
$\Leps(u_0)$ defined by~\eqref{eq:correcteur}. 
Except for Lemma~\ref{lem:dual_Neu2}, these estimates are classical; we provide (short) proofs since we
additionally track the direct dependency of the constants on the parameters $\alpha$ and $\beta$
characterizing the spectrum of $\mathbb{A}$ and on the various length scales 
present in the problem.

\subsection{Dual-norm estimates} \label{ap:dual}

Let $D$ be an open, connected, polytopal subset of $\Omega$; in this work,
we will need the cases where $D=\Omega$ or where $D=T\in \TH$. Let $\ell_D$ be a length scale associated with $D$, e.g., its diameter. 
Our goal is to bound the dual norm of the linear map such that
\begin{equation} \label{eq:def_calF}
w\mapsto \mathcal{F}_\eps(w):=\int_D\left(\mathbb{A}_\eps\grad\Leps(u_0)-\mathbb{A}_0\grad u_0\right){\cdot}\grad w,
\end{equation} 
for all $w\in H^1_0(D)$ (Dirichlet case), or for all $w\in H^1_\star(D)\define \{w\in H^1(D)\mid \int_Dw=0\}$ (Neumann case); note that $\mathcal{F}_\eps(w)$ does not change if the values of $w$ are shifted by a constant.

\begin{lemma}[Dual norm, Dirichlet case] \label{lem:dual_Dir}
Assume that the homogenized solution $u_0$ belongs to $H^2(D)$ and that, for any $1\leq l\leq d$, the corrector $\mu_{l}$ belongs to $W^{1,\infty}(\R^d)$. Then,
\begin{equation} \label{eq:dual_Dir}
\sup_{w\in H^1_0(D)} \frac{\left|\mathcal{F}_\eps(w)\right|}{\|\grad w\|_{L^2(D)^d}} 
\leq c\, \beta \eps \senorm{H^2(D)}{u_0},
\end{equation}
with $c$ independent of $\eps$, $D$, $u_0$, $\alpha$, $\beta$, and possibly depending on 
$d$, and on $\displaystyle\max_{1\leq l\leq d}\norm{W^{1,\infty}(\R^d)}{\mu_l}$. 
\end{lemma}

\begin{proof}
For any integer $1\leq i\leq d$, we have
  \begin{align}
      {\left[\mathbb{A}_\eps\grad\Leps(u_0)\right]}_i&=\sum_{j=1}^d{\left[\mathbb{A}_\eps\right]}_{ij}\partial_j\Leps(u_0) \nonumber\\
      &=\sum_{j=1}^d{\left[\mathbb{A}_\eps\right]}_{ij}\left(\partial_j u_0+\eps\sum_{l=1}^d\left(\frac{1}{\eps}\Reps(\partial_j\mu_l)\partial_l u_0+\Reps(\mu_l)\partial^2_{j,l} u_0\right)\right) \nonumber\\
      &={\left[\mathbb{A}_0\grad u_0\right]}_i+\sum_{l=1}^d\Reps(\theta_i^l)\partial_l u_0+\eps\sum_{l,j=1}^d{\left[\mathbb{A}_\eps\right]}_{ij}\Reps(\mu_l)\partial^2_{j,l} u_0, \label{eq:pr.app.2}
  \end{align}
with $\theta_i^l\define\mathbb{A}_{il}+\sum_{j=1}^d\mathbb{A}_{ij}\partial_j\mu_l-{\left[\mathbb{A}_0\right]}_{il}$ satisfying the following properties:
\begin{itemize}
  \item[$\bullet$] $\theta_i^l\in L^{\infty}_{\rm per}(Q)$ by assumption on $\mathbb{A}$ and on the correctors $\mu_{l}$;
  \item[$\bullet$] $\int_Q\theta_i^l=0$ as a consequence of~\eqref{eq:Azer};
  \item[$\bullet$] $\sum_{i=1}^d\partial_i\theta^l_i=0$ in $\R^d$ as a consequence of~\eqref{eq:corr}.
\end{itemize}
Adapting~\cite[Equation~(1.11)]{JiKoO:94} (see also~\cite[Sections I.3.1 and I.3.3]{GiRav:86}), we infer that, for any integer $1\leq l\leq d$, there exists a skew-symmetric matrix $\mathbb{T}^l\in W^{1,\infty}_{\rm per}(Q)^{d\times d}$, satisfying $\int_Q\mathbb{T}^l=\mathbb{0}$ and such that, for any integer $1\leq i\leq d$,
\begin{equation} \label{eq:proof.5}
  \theta^l_i=\sum_{q=1}^d\partial_q\mathbb{T}^l_{qi}.
\end{equation}
Plugging~\eqref{eq:proof.5} into~\eqref{eq:pr.app.2}, we infer that, for any integer $1\leq i\leq d$,
\begin{multline*}
{\left[\mathbb{A}_\eps\grad\Leps(u_0)\right]}_i-{\left[\mathbb{A}_0\grad u_0\right]}_i=\eps\left(\sum_{l,q=1}^d\partial_q(\Reps(\mathbb{T}^l_{qi}))\partial_l u_0\right.\\\left. +\sum_{l,j=1}^d{\left[\mathbb{A}_\eps\right]}_{ij}\Reps(\mu_l)\partial^2_{j,l} u_0\right).
\end{multline*}
Since $\partial_q(\Reps(\mathbb{T}^l_{qi}))\partial_l u_0=\partial_q(\Reps(\mathbb{T}^l_{qi})\partial_l u_0)-\Reps(\mathbb{T}^l_{qi})\partial^2_{q,l} u_0$, and recalling the definition~\eqref{eq:def_calF} of $\mathcal{F}_\eps$, this yields
\begin{align}
\mathcal{F}_\eps(w) = {}&\eps\bigg(\sum_{i,l,j=1}^d\int_D{\left[\mathbb{A}_\eps\right]}_{ij}\Reps(\mu_l)\partial^2_{j,l} u_0\,\partial_iw -\sum_{i,l,q=1}^d\int_D\Reps(\mathbb{T}^l_{qi})\partial^2_{q,l} u_0\,\partial_iw\bigg)\nonumber \\ &+\eps\sum_{i,l,q=1}^d\int_D\partial_q\left(\Reps(\mathbb{T}^l_{qi})\partial_l u_0\right)\partial_iw.\label{eq:proof_calF1}
\end{align}
Since $\mathbb{T}^l_{qi}=-\mathbb{T}^l_{iq}$ for any integers $1\leq i,q\leq d$, we infer by integration by parts of the last term that
\begin{align}
\mathcal{F}_\eps(w) ={}&\eps\bigg(\sum_{i,l,j=1}^d\int_D{\left[\mathbb{A}_\eps\right]}_{ij}\Reps(\mu_l)\partial^2_{j,l} u_0\,\partial_iw-\sum_{i,l,q=1}^d\int_D\Reps(\mathbb{T}^l_{qi})\partial^2_{q,l} u_0\,\partial_iw\bigg) \nonumber \\ &+ \eps\sum_{i,l,q=1}^d\int_{\partial D}\partial_q\left(\Reps(\mathbb{T}^l_{qi})\partial_l u_0\right)n_{\partial D,i}\,w,\label{eq:proof_calF2}
\end{align}
where $\vec{n}_{\partial D}$ is the unit outward normal to $D$.
Since $w\in H^1_0(D)$, we obtain
\begin{equation*}
\mathcal{F}_\eps(w)=\eps\bigg(\sum_{i,l,j=1}^d\int_D{\left[\mathbb{A}_\eps\right]}_{ij}\Reps(\mu_l)\partial^2_{j,l} u_0\,\partial_iw-\sum_{i,l,q=1}^d\int_D\Reps(\mathbb{T}^l_{qi})\partial^2_{q,l} u_0\,\partial_iw\bigg).
\end{equation*}
Using the Cauchy--Schwarz inequality, we finally deduce that
\begin{equation*} 
  \sup_{w\in H^1_0(D)}\frac{\left|\mathcal{F}_\eps(w)\right|}{\norm{L^2(D)^d}{\grad w}}\leq c\,\beta\eps\max_{1\leq l\leq d}\left(\norm{L^{\infty}(\R^d)}{\mu_l},\beta^{-1}\norm{L^{\infty}(\R^d)^{d\times d}}{\mathbb{T}^l}\right)\senorm{H^2(D)}{u_0}.
\end{equation*}
We conclude by observing that $\norm{L^{\infty}(\R^d)^{d\times d}}{\mathbb{T}^l} \leq c\,\max_{1\le i\le d}\norm{L^{\infty}(\R^d)}{\theta_i^l} \leq c\,\beta$.
\end{proof}

\begin{lemma}[Dual norm, Neumann case (i)] \label{lem:dual_Neu1}
Assume that the homogenized solution $u_0$ belongs to $H^2(D)\cap W^{1,\infty}(D)$ and that, for any $1\leq l\leq d$, the corrector $\mu_{l}$ belongs to $W^{1,\infty}(\R^d)$. Then,
\begin{equation}\label{eq:dual_Neu1}
\sup_{\substack{w\in H^1_\star(D)}} \frac{\left|\mathcal{F}_\eps(w)\right|}{\|\grad w\|_{L^2(D)^d}} 
\leq c\, \beta \left(\eps \senorm{H^2(D)}{u_0} + |\partial D|^{\nicefrac{1}{2}}\eps^{\nicefrac12} \senorm{W^{1,\infty}(D)}{u_0}\right),
\end{equation}
with $c$ independent of $\eps$, $D$, $u_0$, $\alpha$, $\beta$, and possibly depending on $d$, and on $\displaystyle\max_{1\leq l\leq d}\norm{W^{1,\infty}(\R^d)}{\mu_l}$.  
\end{lemma}

\begin{proof}
Our starting point is~\eqref{eq:proof_calF1}.
The first two terms in the right-hand side are responsible for a contribution of order $\beta\eps\senorm{H^2(D)}{u_0}$, and it only remains to bound the last term.
Following the ideas of~\cite[p.~29]{JiKoO:94}, we define, for $\eta>0$, the domain $D_\eta\define\left\{\vec{x}\in D\mid{\rm dist}(\vec{x},\partial D)<\eta\right\}$. If $\eta$ is above a critical value (which scales as $\ell_D$), $D_\eta=D$, otherwise $D_\eta\subsetneq D$. We introduce the cut-off function $\zeta_\eta\in C^0(\overline{D})$ such that $\zeta_\eta\equiv 0$ on $\partial D$, defined by $\zeta_\eta(\vec{x})={\rm dist}(\vec{x},\partial D)/\eta$ if $\vec{x}\in D_\eta$, and $\zeta_\eta(\vec{x})=1$ if $\vec{x}\in D\setminus D_\eta$. We have $0\leq \zeta_\eta\leq 1$ and $\max_{1\leq q\leq d}\norm{L^\infty(D)}{\partial_q\zeta_\eta}\leq\eta^{-1}$. We first infer that
\begin{equation*} 
    \eps\sum_{i,l,q=1}^d\int_D\partial_q\left(\Reps(\mathbb{T}^l_{qi})\partial_l u_0\right)\partial_iw=\eps\sum_{i,l,q=1}^d\int_{D_\eta}\partial_q\left((1-\zeta_\eta)\Reps(\mathbb{T}^l_{qi})\partial_l u_0\right)\partial_iw,
\end{equation*}
since $(1-\zeta_\eta)$ vanishes identically on $D\setminus D_\eta$ and since
$$\sum_{i,l,q=1}^d\int_{D}\partial_q\left(\zeta_\eta\Reps(\mathbb{T}^l_{qi})\partial_l u_0\right)\partial_iw=0$$
as can be seen by integration by parts, using the fact that $\mathbb{T}^l_{qi}=-\mathbb{T}^l_{iq}$ for any integers $1\leq i,q\leq d$, and the fact that $\zeta_\eta$ vanishes identically on $\partial D$. Then, accounting for the fact that
\begin{multline*}
\eps\,\partial_q\left((1-\zeta_\eta)\Reps(\mathbb{T}^l_{qi})\partial_l u_0\right)=-\eps\,\partial_q\zeta_\eta\,\Reps(\mathbb{T}^l_{qi})\partial_l u_0\\+(1-\zeta_\eta)\Reps\left(\partial_q\mathbb{T}^l_{qi}\right)\partial_lu_0+\eps(1-\zeta_\eta)\Reps(\mathbb{T}^l_{qi})\partial^2_{q,l}u_0,
\end{multline*}
we infer that
\begin{multline*}
    \hspace{-0.25cm}\left|\eps\!\!\!\sum_{i,l,q=1}^d\int_D\partial_q\left(\Reps(\mathbb{T}^l_{qi})\partial_l u_0\right)\partial_iw\right| \leq c\bigg[|D_\eta|^{\nicefrac12}\left(\frac{\eps}{\eta}+1\right)\left(\max_{1\leq l\leq d}\norm{W^{1,\infty}(\R^d)^{d\times d}}{\mathbb{T}^l}\right)\\\senorm{W^{1,\infty}(D)}{u_0}+\eps\left(\max_{1\leq l\leq d}\norm{L^{\infty}(\R^d)^{d\times d}}{\mathbb{T}^l}\right)\senorm{H^2(D)}{u_0}\bigg]\norm{\L[D]^d}{\grad w}.
  \end{multline*}
Using the estimate $|D_\eta|\le \eta |\partial D|$, the fact that $\max_{1\leq l\leq d}\norm{W^{1,\infty}(\R^d)^{d\times d}}{\mathbb{T}^l}\leq c\,\beta$, and since the function $\eta\mapsto\frac{\eps}{\sqrt{\eta}}+\sqrt{\eta}$ is minimal for $\eta=\eps$, we finally infer the bound~\eqref{eq:dual_Neu1}. 
\end{proof}

\begin{remark}[Weaker regularity assumption] \label{re:mhm}
Without the regularity assumption $u_0\in W^{1,\infty}(D)$, one can still invoke a Sobolev embedding since $u_0\in H^2(D)$. The second term between the parentheses in the right-hand side of~\eqref{eq:dual_Neu1} becomes
$$c(p)\big(|\partial D|\eps \ell_{D}^{-d}\big)^{\nicefrac{1}{2}-\nicefrac{1}{p}}(\senorm{H^1(D)}{u_0}+\ell_D\senorm{H^2(D)}{u_0}),
$$
where $p=6$ for $d=3$ and $p$ can be taken as large as wanted for $d=2$ (note that $c(p)\to+\infty$ when $p\to+\infty$ in that case). The derivation of estimates in this setting is considered in~\cite{PaVaV:17}. Therein, the authors claim that their method is able to get rid of the resonance error (without oversampling). We believe there is an issue with the bound~\cite[eq.~(27)]{PaVaV:17}, which should exhibit the resonant contribution ${\left(\nicefrac{\eps}{\ell_{D}}\right)}^{\nicefrac{1}{2}-\nicefrac{1}{p}}\senorm{H^1(D)}{u_0}$.
\end{remark}

\begin{lemma}[Dual norm, Neumann case (ii)] \label{lem:dual_Neu2}
Assume that $D=T\in \TH$ where $\TH$ is a member of an admissible mesh sequence in the sense of Definition~\ref{def:adm}; set $\ell_D=H_T$. 
Assume that the homogenized solution $u_0$ belongs to $H^3(D)$ and that there is $\kappa>0$ so that $\mathbb{A}\in C^{0,\kappa}(\R^d;\R^{d\times d})$. Then,
\begin{multline}\label{eq:dual_Neu2}
\sup_{w\in H^1_\star(D)} \frac{\left|\mathcal{F}_\eps(w)\right|}{\|\grad w\|_{L^2(D)^d}} 
\leq c\, \beta \left(\left(\eps+(\eps\ell_D)^{\nicefrac12}\right) \senorm{H^2(D)}{u_0} + \eps\ell_D\senorm{H^3(D)}{u_0} \right.\\\left. + \eps^{\nicefrac12}\ell_D^{-\nicefrac12}\senorm{H^1(D)}{u_0}\right),
\end{multline}
with $c$ independent of $\eps$, $D$, $u_0$, $\alpha$, $\beta$, and possibly depending on $d$, $\gamma$, and $\norm{C^{0,\kappa}(\R^d;\R^{d\times d})}{\mathbb{A}/\beta}$.
\end{lemma}

\begin{proof}
We proceed as in the proof of Lemma~\ref{lem:dual_Dir}. Concerning the regularity of
$\theta_i^l$, we now have $\theta_i^l\in C^{0,\iota}(\R^d)$ for some $\iota>0$ since the H\"older continuity of $\mathbb{A}$ on $\R^d$ implies the H\"older continuity of $\mu_{l}$ and $\grad\mu_{l}$ on $\R^d$ for any $1\leq l\leq d$; cf., e.g.,~\cite[Theorem 8.22 and Corollary 8.36]{GiTru:01}. Following \cite[p.~6-7]{JiKoO:94} and \cite[p.~131-132]{LBLLo:13}, we infer that the skew-symmetric matrix 
$\mathbb{T}^l$ is such that $\mathbb{T}^l\in C^1(\R^d)^{d\times d}$. 
Our starting point is~\eqref{eq:proof_calF2}.
The first two terms in the right-hand side are responsible for a contribution of order $\beta\eps\senorm{H^2(D)}{u_0}$, and it only remains to bound the last term. We have
\begin{multline*} 
\eps\sum_{i,l,q=1}^d\int_{\partial D}\partial_q\left(\Reps(\mathbb{T}^l_{qi})\partial_l u_0\right)n_{\partial D,i}\,w=\eps\sum_{i,l,q=1}^d\int_{\partial D}\Reps(\mathbb{T}^l_{qi})\partial^2_{q,l} u_0\,n_{\partial D,i}\,w \\
+\sum_{i,l,q=1}^d\int_{\partial D}\Reps\left(\partial_q\mathbb{T}^l_{qi}\right)\partial_l u_0\,n_{\partial D,i}\,w\enifed \term_1+\term_2.
\end{multline*}
Using the Cauchy--Schwarz inequality and the trace inequality~\eqref{eq:tr.c}, the first term in the right-hand side can be estimated as
$$\left|\term_1\right|\leq c\, \beta\eps\ell_D^{-1}\left(\senorm{H^2(D)}{u_0}+\ell_D\senorm{H^3(D)}{u_0}\right)\left(\norm{L^2(D)}{w}+\ell_D\norm{L^2(D)^d}{\grad w}\right),
$$
since $\max_{1\leq l\leq d}\norm{C^0(\R^d)^{d\times d}}{\mathbb{T}^l}\leq c\,\beta$. Observing that $\int_Dw=0$, we can use the Poincar\'e inequality~\eqref{eq:poin} to infer that
$$
\left|\term_1\right|\leq c\, \beta\eps\left(\senorm{H^2(D)}{u_0}+\ell_D\senorm{H^3(D)}{u_0}\right)\norm{L^2(D)^d}{\grad w}.
$$
To estimate the second term in the right-hand side, we adapt the ideas from~\cite[Lemma 4.6]{LBLLo:13}. Considering the matching simplicial sub-mesh of $D$, let us collect in the set $\mathfrak{F}_D$ all the sub-faces composing the boundary of $D$. 
Then, we can write
$$
\term_2=\sum_{\sigma\in\mathfrak{F}_D}\sum_{l=1}^d\sum_{q=1}^d\sum_{q<i\leq d}\int_\sigma\Reps\left(\grad\mathbb{T}^l_{qi}\right){\cdot}\vec{\tau}_\sigma^{qi}\,\partial_lu_0\,w,
$$
where the vectors $\vec{\tau}_\sigma^{qi}$ are such that $\|\vec{\tau}_\sigma^{qi}\|_{\ell^2}\leq 1$ and $\vec{\tau}_\sigma^{qi}{\cdot}\vec{n}_{\partial D\mid\sigma}=0$. Then, using a straightforward adaptation of the result in~\cite[Lemma 4.6]{LBLLo:13}, and since $\max_{1\leq l\leq d}\norm{C^1(\R^d)^{d\times d}}{\mathbb{T}^l}\leq c\,\beta$, we infer that
\begin{multline*}
\left| \int_\sigma\Reps\left(\grad\mathbb{T}^l_{qi}\right){\cdot}\vec{\tau}_\sigma^{qi}\,\partial_lu_0\,w \right| \leq c\, \beta\eps^{\nicefrac12} H_S^{-\nicefrac32}\left(\senorm{H^1(S)}{u_0}+H_S\senorm{H^2(S)}{u_0}\right)\\\left(\norm{L^2(S)}{w}+H_S\norm{L^2(S)^d}{\grad w}\right),
\end{multline*}
where $S$ is the simplicial sub-cell of $D$ having $\sigma$ as face. Collecting the contributions of all the sub-faces $\sigma \in \mathfrak{F}_D$ and using the mesh regularity assumptions on $D$, we infer that
\[
\left|\term_2\right| \leq c\, \beta\eps^{\nicefrac12} \ell_D^{-\nicefrac32}\left(\senorm{H^1(D)}{u_0}+\ell_D\senorm{H^2(D)}{u_0}\right)\left(\norm{L^2(D)}{w}+\ell_D\norm{L^2(D)^d}{\grad w}\right).
\]
Finally, invoking the Poincar\'e inequality~\eqref{eq:poin} since $w$ has zero mean-value in $D$ yields 
\[
\left|\term_2\right| \leq c\, \beta\eps^{\nicefrac12} \ell_D^{-\nicefrac12}\left(\senorm{H^1(D)}{u_0}+\ell_D\senorm{H^2(D)}{u_0}\right)\norm{L^2(D)^d}{\grad w}.
\]
Collecting the above bounds on $\term_1$ and $\term_2$ concludes the proof. 
\end{proof}

\subsection{Global energy-norm estimate} \label{ap:ts}

\begin{lemma}[Energy-norm estimate] \label{le:ts}
  Assume that the homogenized solution $u_0$ belongs to $H^2(\Omega)\cap W^{1,\infty}(\Omega)$, and that, for any $1\leq l\leq d$, the corrector $\mu_{l}$ belongs to $W^{1,\infty}(\R^d)$. Then,
\begin{equation} \label{eq:ts_global}
\norm{\L^d}{\mathbb{A}_\eps^{\nicefrac12}\grad (u_\eps-\Leps(u_0))}\leq c\,\beta^{\nicefrac12}\left(\rho^{\nicefrac12}\eps\,\senorm{H^2(\Omega)}{u_0}+{|\partial\Omega|}^{\nicefrac12}\eps^{\nicefrac12}\,\senorm{W^{1,\infty}(\Omega)}{u_0}\right),
\end{equation}
with $c$ independent of $\eps$, $\Omega$, $u_0$, $\alpha$, $\beta$, and possibly depending on $d$, and on $\displaystyle\max_{1\leq l\leq d}\norm{W^{1,\infty}(\R^d)}{\mu_l}$. 
\end{lemma}

\begin{proof}
The regularity assumptions on $u_0$ and the correctors imply
$(u_\eps-\Leps(u_0))\in H^1(\Omega)$;
however, we do not have $(u_\eps-\Leps(u_0))\in H^1_0(\Omega)$.
Following the ideas in~\cite[p.~28]{JiKoO:94}, 
we define, for $\eta>0$, the domain $\Omega_\eta\define\left\{\vec{x}\in\Omega\mid{\rm dist}(\vec{x},\partial\Omega)<\eta\right\}$.
If $\eta$ is above a critical value, $\Omega_\eta=\Omega$, otherwise $\Omega_\eta\subsetneq\Omega$.
We introduce the cut-off function $\zeta_\eta\in C^0(\overline{\Omega})$ such that $\zeta_\eta\equiv 0$ on $\partial\Omega$, defined by $\zeta_\eta(\vec{x})={\rm dist}(\vec{x},\partial\Omega)/\eta$ if $\vec{x}\in\Omega_\eta$, and $\zeta_\eta(\vec{x})=1$ if $\vec{x}\in\Omega\setminus\Omega_\eta$. We have
$0\leq \zeta_\eta\leq 1$ and 
$\displaystyle\max_{1\leq i\leq d}\norm{L^\infty(\Omega)}{\partial_i\zeta_\eta}\leq \eta^{-1}$.
The function $\zeta_\eta$ allows us to define a corrected first-order two-scale expansion ${\cal L}^{1,0}_\eps(u_0)\define u_0+\eps\zeta_\eta\sum_{l=1}^d\Reps(\mu_l)\partial_lu_0$ such that $(u_\eps-{\cal L}^{1,0}_\eps(u_0))\in\Hoz$.
We start with the triangle inequality:
\begin{align} 
  \norm{\L^d}{\mathbb{A}_\eps^{\nicefrac12}\grad(u_\eps-\Leps(u_0))}\leq{}&\norm{\L^d}{\mathbb{A}_\eps^{\nicefrac12}\grad(u_\eps-{\cal L}^{1,0}_\eps(u_0))}\nonumber \\&+\norm{\L^d}{\mathbb{A}_\eps^{\nicefrac12}\grad(\Leps(u_0)-{\cal L}^{1,0}_\eps(u_0))}.\label{eq:proof.1}
\end{align}
Let us focus on the first term in the right-hand side of~\eqref{eq:proof.1}. We have
\begin{multline*}
\norm{\L^d}{\mathbb{A}_\eps^{\nicefrac12}\grad(u_\eps-{\cal L}^{1,0}_\eps(u_0))}^2 = \int_\Omega\mathbb{A}_\eps\grad\left(u_\eps-\Leps(u_0)\right){\cdot}\grad\left(u_\eps-{\cal L}^{1,0}_\eps(u_0)\right)\\+\int_\Omega\mathbb{A}_\eps\grad\left(\Leps(u_0)-{\cal L}^{1,0}_\eps(u_0)\right){\cdot}\grad\left(u_\eps-{\cal L}^{1,0}_\eps(u_0)\right).
\end{multline*}
Since $(u_\eps-{\cal L}^{1,0}_\eps(u_0))\in\Hoz$, we infer that
\begin{align} 
  \norm{\L^d}{\mathbb{A}_\eps^{\nicefrac12}\grad(u_\eps-{\cal L}^{1,0}_\eps(u_0))}\leq{}&\alpha^{-\nicefrac12}\sup_{w\in\Hoz}\frac{\left|\int_\Omega\mathbb{A}_\eps\grad\left(u_\eps-\Leps(u_0)\right){\cdot}\grad w\right|}{\norm{\L^d}{\grad w}} \nonumber \\&+\norm{\L^d}{\mathbb{A}_\eps^{\nicefrac12}\grad(\Leps(u_0)-{\cal L}^{1,0}_\eps(u_0))}.\label{eq:proof.2}
\end{align}
Since $\int_\Omega\mathbb{A}_\eps\grad u_\eps{\cdot}\grad w=\int_\Omega\mathbb{A}_0\grad u_0{\cdot}\grad w$ for any $w\in\Hoz$ in view of~\eqref{eq:osc} and~\eqref{eq:hom}, the estimates~\eqref{eq:proof.1} and~\eqref{eq:proof.2} lead to
\begin{align} 
  \norm{\L^d}{\mathbb{A}_\eps^{\nicefrac12}\grad(u_\eps-\Leps(u_0))}\leq{}&\alpha^{-\nicefrac12}\sup_{w\in\Hoz}\frac{\left|\mathcal{F}_\eps(w)\right|}{\norm{\L^d}{\grad w}} \nonumber\\&+2\beta^{\nicefrac12}\norm{\L^d}{\grad(\Leps(u_0)-{\cal L}^{1,0}_\eps(u_0))},\label{eq:proof.3}
\end{align}
recalling that $\mathcal{F}_\eps(w)=\int_\Omega\left(\mathbb{A}_\eps\grad\Leps(u_0)-\mathbb{A}_0\grad u_0\right){\cdot}\grad w$. Since we can bound the first term in the right-hand side of~\eqref{eq:proof.3} using Lemma~\ref{lem:dual_Dir} (with $D=\Omega$), it remains to estimate the second term. Owing to the definition of $\zeta_\eta$, we infer that
\begin{equation} \label{eq:proof.7}
  \norm{\L^d}{\grad(\Leps(u_0)-{\cal L}^{1,0}_\eps(u_0))}=\eps\norm{L^2(\Omega_\eta)^d}{\grad\left((1-\zeta_\eta)\sum_{l=1}^d\Reps(\mu_l)\partial_lu_0\right)}.
\end{equation}
For any integer $1\leq i\leq d$, we have
\begin{multline*} 
  \partial_i\left((1-\zeta_\eta)\sum_{l=1}^d\Reps(\mu_l)\partial_lu_0\right)=-\partial_i\zeta_\eta\sum_{l=1}^d\Reps(\mu_l)\partial_lu_0\\+\frac{(1-\zeta_\eta)}{\eps}\sum_{l=1}^d\Reps(\partial_i\mu_l)\partial_lu_0+(1-\zeta_\eta)\sum_{l=1}^d\Reps(\mu_l)\partial^2_{i,l}u_0,
\end{multline*}
and using the properties of the cut-off function $\zeta_\eta$, we infer that
\begin{multline*} 
  \eps\norm{L^2(\Omega_\eta)^d}{\grad\left((1-\zeta_\eta)\sum_{l=1}^d\Reps(\mu_l)\partial_lu_0\right)}\leq c \bigg({|\Omega_\eta|}^{\nicefrac12}\left(\frac{\eps}{\eta}+1\right)\senorm{W^{1,\infty}(\Omega)}{u_0}\\+\eps\senorm{H^2(\Omega)}{u_0}\bigg).
\end{multline*}
Since $|\Omega_\eta|\leq |\partial\Omega|\eta$, 
and choosing $\eta=\eps$ to minimize the function $\eta\mapsto\frac{\eps}{\sqrt{\eta}}+\sqrt{\eta}$, we can conclude the proof (note that $\rho\geq1$ by definition).
\end{proof}


\footnotesize
\bibliographystyle{plain}
\bibliography{mshho}

\end{document}